\documentclass[11pt]{amsart}

\usepackage{verbatim,amssymb,color,rotating,ulem,amsmath}
\usepackage{graphicx}
\input xy
\xyoption{all}

\newtheorem{thm}{Theorem}[section]
\newtheorem{prop}[thm]{Proposition}
\newtheorem{cor}[thm]{Corollary}
\newtheorem{lemma}[thm]{Lemma}
\newtheorem{rk}[thm]{Remark}

\theoremstyle{definition}
\newtheorem{ex}[thm]{Example}
\newenvironment{pf}{\begin{proof}}{\end{proof}}

\newcommand{\s}{\smallskip}

\newcommand{\tmod}[1]{\; (#1)}

\newcommand{\Z}{\mathbb{Z}}

\newcommand{\cudz}{ ,\hspace{-0.37mm},}

\newcommand{\zz}{\mathbb{Z}}

\title[Periodic self-homeomorphisms of bordered surfaces]{On topological classification
of finite cyclic  actions on bordered surfaces}
\author{G. Gromadzki, S. Hirose, B. Szepietowski}
\address{G. Gromadzki, B. Szepietowski: Institute of Mathematics, Faculty of Mathematics, Physics and Informatics, University of Gda\'nsk, 80-308 Gda\'nsk, Poland \vspace{-4mm}}
\address{S. Hirose: Department of Mathematics, Faculty of Science and Technology, Tokyo University of Science, Noda, Chiba, 278-8510, Japan }
\email{grom@mat.ug.edu.pl, hirose\b{ }susumu@ma.noda.tus.ac.jp, blaszep@mat.ug.edu.pl}

\thanks{G. Gromadzki and B. Szepietowski supported by National Science Centre, Poland by grant NCN
2015/17/B/ST1/03235.
S. Hirose supported by Grant-in-Aid for Scientific Research (C) (No.16K05156), Japan Society for the Promotion of Science. This paper was initiated, and
partially written, when G. Gromadzki visited Max-Planck Mathematical Institute in Bonn and Tokyo University of Science.}

\begin{document}
\begin{abstract}
In \cite{Hirose}, {the second author}   showed that, except for a few cases, the
order $N$ of a cyclic group of self-homeomorphisms of a closed
orientable topological surface $S_g$ of genus $g \geq 2$ determines
the group up to a topological conjugation, provided that $N\geq 3g$. The first author et al.  undertook in \cite{BCGH} a more general problem of topological classification of
such group actions for  $N>2(g-1)$. In \cite{GS} we considered the analogous problem for closed non-orientable surfaces, and
 in \cite{GSZ} - the problem of classification of cyclic actions generated by an orientation reversing self-homeomorphism.
The present paper, in which we
deal with topological classification of  actions on bordered surfaces of finite cyclic
groups of order $N>p-1$, where $p$ is the
algebraic genus of the surface, completes our project of topological  classification of {\it \cudz large"}  cyclic actions on compact surfaces.
We apply obtained results to solve the problem of  uniqueness of the actions realising the solutions of the so called minimum genus and maximum order problems for bordered surfaces found in \cite{BEGG}.
\end{abstract}
\maketitle

\section{Introduction}
{By an  action of a  group $G$ on a  surface $S$ we understand an embedding of $G$ into the group $\mathrm{Homeo}(S)$ of homeomor\-phisms of $S$. Two such actions are {\it topologically conjugate}, or of the same {\it topological type}, if the images of $G$ are conjugate in $\mathrm{Homeo}(S)$.}

In \cite{Hirose} {it was shown} that, except for a few cases, the order
$N$ of the finite cyclic group $\zz_N$ acting on a closed orientable
topological surface $S_g$ of genus $g \geq 2$ determines the topological type of the action, provided that $N\geq 3g$. In \cite{BCGH} the first author {et al.} undertook a more general problem of topological classification of such
actions for $N>2(g-1)$.
This is an essential extension, because between $3g$
and $4g+2$  only $3g+1,3g+2,3g+4$ and $4g$ can stand as the
period of a single self-homeomorphism of $S_g$,  whereas there are infinitely many
rational values of $a,b$ such that for $N=ag+b$ we have $N>2g-2$ and $N$ is the
period of a self-homeomorphisms of $S_g$ for infinitely many $g$.  In \cite{GSZ} we  considered analogous  problem for cyclic actions generated by an orientation reversing self-homeomorphism,
while in \cite{GS} - a similar problem for closed non-orientable surfaces, obtaining a classification of topological types of action of $\zz_N$ on a surface $S$  in function of a possible type of the quotient
orbifold $S/\Z_N$, provided that $N$ is sufficiently big.

\smallskip
The present paper, in which we
deal with topological classification of  actions of $\zz_N$ on a bordered surface of
algebraic genus $p$, where $N>p-1$, completes our project of topological  classification of big cyclic actions on compact surfaces.
The lower bound $p-1$ for the order of an action is essential for two  reasons.
The first is that we again cover a quite  large class of actions,
since there are infinitely many rational values of $a,b$, for which
there are infinitely many values of $p$, such that a bordered surface of algebraic genus $p$ admits a cyclic action of order $N=ap+b$ and $N>p-1$.
 The second reason for the bound $N>p-1$ is that it is
satisfied for all the actions realising the solutions of the so called minimum genus
and maximum order problems for  bordered surfaces found in \cite{BEGG}, and the
 question about their topological rigidity partially motivated the present  paper.

\smallskip
{
Our results can be seen as a  topological classification of cyclic group actions of order $N$ on  bordered surfaces of algebraic genus  $p \leq N$. Another  problem, suggested by the referee of this paper, would be to obtain a similar classification for surfaces of large  genera i.e.  $p$ bigger than $N$.
We remark that there are many results in the literature  about the spectrum of genera of surfaces admitting a given finite group  as a group  of self-homeomorphisms.  As an example of such results in the case of  closed orientable surfaces let us mention the important paper \cite{Kulk} by Kulkarni. While we believe that it shouldn't be difficult to obtain similar results for bordered surfaces, it seems that it would be a rather difficult problem to classify topologically actions of order $N<p$. The main reason is that the orbit spaces which occur in the case $N<p$ may have much bigger and much more complicated mapping class groups (see section \ref{ss:MCG} for a definition) than for $p \leq N$,  and in such a case our method, based on a good understanding of these mapping class groups, is not effective.}

\smallskip There are two more interesting features of the  actions considered in this paper.
The first is that finite group actions on compact surfaces of negative Euler characteristic  may be realised by analytic  actions on Riemann surfaces, or dianalytic actions on  Klein surfaces, due to the Hurwitz-Nielsen,
Kerkjarto  and Alling-Greenleaf geometrizations mentioned in Subsection \ref{geometrization}.
The loci in the moduli spaces of Klein surfaces composed of the points classifying  the surfaces dianalitically realizing the actions considered here have dimensions $1,2$ or $3$ (this follows from  Lemma \ref{big o-r}, formula (\ref{can-fuchs}) {and classical formula of Fricke and Klein for dimension of Teichm\"uller spaces of Fuchsian groups c.f. Theorem 0.3.2 in \cite{BEGG})}.
This is  similar as in the case of actions on unbordered non-orientable  Klein surfaces \cite{GS},   or orientation reversing automorphisms of classical Riemann surfaces \cite{GSZ}, but  in contrast
 to the classical case of  orientation-preserving cyclic actions of order $>2g-2$ described in \cite{BCGH}, where the loci of such structures in the moduli space  are $0$-dimensional, which means, in particular, that  the topological type of an action of an orientation-preserving self-homeomorphism of such order usually uniquely determines  the conformal type of a Riemann surface on which it acts as an automorphism. Finally, observe that our results can be stated in terms of birational actions on real algebraic curves due to the functorial equivalence between bordered Klein surfaces and such curves described in \cite{AG}.

\smallskip
{This paper is organized as follows. In Section \ref{sec:statement} we state our main results. Section \ref{sec:preli} contains
necessary preliminaries concerning finite topological actions on bordered surfaces from the combinatorial point of view. In particular we review non-euclidean crystallographic groups. In Section \ref{sec:types}
we determine  the possible topological types of the orbit space (orbifold) of a cyclic  action of order $N$ on a bordered surface of algebraic  genus $p<N+1$. We obtain ten different topological  types here, all of which are either a disc or an annulus or a M\"obius band, with some cone points in the interior and some corner points on the boundary.
%Listing all involved here actions is a relatively easy task but their topological classification is a bit more involved
In Section \ref{ss:MCG} we review the relationship between the groups of automorphisms of non-euclidean crystallographic groups and mapping class groups. We also compute the mapping class groups of three surfaces:
 once-punctured annulus, once-punctured M\"obius band and  twice-punctured disc, which are needed for Section \ref{sec:main}, where we prove our main results. %Section 5 is divided into ten subsections corresponding to different topological types of the orbit space. In each of these subsections we classify topological types of action of a cyclic group on a bordered surface with the given quotient orbifold. The obtained results are too  numerous to  be presented in the introduction, and therefore we  restrict ourselves here to   present only two samples.
Finally, in Section \ref{sec:minmax} we apply our results to study topological rigidity of %actions realizing
the solutions of the so-called minimum genus and maximum order problems for cyclic actions on bordered surfaces, solved over 30 years ago in \cite{BEGG}.
}

\section{Statement of the main results}\label{sec:statement}
Suppose that a cyclic group of order $N$ acts on a bordered surface $S$ of algebraic genus $p$, where $N>p-1$. We show in Section \ref{sec:types} that the orbit space $S/{\Z_N}$ is one of the following orbifolds:
\begin{itemize}
\item[(1)] disc with $6$ corner points,
\item[(2)] annulus  with $2$ corner points,
\item[(3)] M\"obius band  with $2$ corner points,
\item[(4)] $1$-punctured disc with $2$ corner points,
\item[(5)] $1$-punctured disc with $4$ corner points,
\item[(6)] $1$-punctured  M\"obius band,
\item[(7)] $2$-punctured disc,
\item[(8)] $1$-punctured annulus,
\item[(9)] $3$-punctured disc,
\item[(10)] $2$-punctured disc with $2$ corner points.
\end{itemize}
%\grom{
Our classification of cyclic actions of big order is split into ten cases and the results are presented in ten consecutive subsections. Their proofs are given in Section  \ref{sec:main}  
which is also divided in ten subsections with the same titles for the reader's convenience.
%}
Throughout the whole paper $\varphi$ will denote the Euler totient function.
We will also need similar function $\psi$ defined in
\cite{BCGH} as $\psi(1)=1$ and given a prime factorization
$C=p_1^{\alpha_1}\cdots p_r^{\alpha_r}>1$
$$\psi(C)=\prod_{i=1}^r(p_i-2)p_i^{\alpha_i-1}\;$$
Observe the analogy with the Euler function $\varphi$ which
is defined for such $C$ as
$$\varphi(C)=\prod_{i=1}^r(p_i-1)p_i^{\alpha_i-1}.$$

\subsection{Actions with a disc with 6 corner points as the quotient orbifold.}
\begin{thm}\label{d6}
There is an action of a cyclic group of order $N$ on a bordered
surface $S$ with a disc having $6$ corner points as the quotient
orbifold if and only if $N=2$ and $S$ is a $3$-holed sphere.
Furthermore, such action is unique up to topological conjugation.
\end{thm}

\subsection{Actions with annulus with 2 corner points as the quotient orbifold.}
\begin{thm}\label{ann2}
There is an action of a cyclic group of order $N$ on a bordered
surface $S$ with an annulus having two corner points as the
quotient orbifold if and only if $N$ is even
and $S$ is one of
the following surfaces:
\begin{itemize}
\item $N/2$-holed Klein bottle,
\item $N/2$-holed torus, where $N/2$ is odd,
\item $(N/2+1)$-holed projective plane,
\item $(N/2+2)$-holed sphere, where $N/2$ is odd.
\end{itemize}
Furthermore, such action is unique up to topological conjugation for
each of these surfaces.
\end{thm}

\subsection{Actions with M\"obius band with 2 corner points as the quotient orbifold.}
\begin{thm}\label{mb2}
There is an action of a cyclic group of order $N$ on a bordered
surface $S$ having a M\"obius band with $2$ corner points as the
quotient orbifold if and only if $N$ is even
and $S$ is
either $N/2$-holed Klein bottle or $N/2$-holed torus, the latter
being possible only for odd $N/2$.
Furthermore, in both cases the
action is unique up to topological conjugation.
\end{thm}

\subsection{Actions with a 1-punctured disc with 2 corner points as the quotient orbifold.}
\begin{thm}\label{d12}
There is an action of a cyclic group of order $N$ on a bordered
surface $S$ with a disc having one cone point of order $m$ and $2$
corner points as the quotient orbifold
if and only if either
\begin{itemize}
\item $m$ is even, $N=m$ and $S$ is $N/2$-holed projective plane, or
\item $m$ is odd, $N=2m$ and $S$ is $N/2$-holed sphere.
\end{itemize}
Furthermore
in both cases the action is unique up to topological conjugation.
\end{thm}

\subsection{Actions with a 1-punctured disc with 4 corner points as the quotient orbifold.}
\begin{thm}\label{d14}
There is an action of a cyclic group of order $N$ on a bordered
surface with a disc having  one cone point of order $m$ and $4$
corner points as the quotient orbifold
if and only if either
\begin{itemize}
\item $m$ is even, $N=m$ and $S$ is $N$-holed projective plane, or
\item $m$ is odd, $N=2m$ and $S$ is $N$-holed sphere.
\end{itemize}
Furthermore, in both cases the action is unique up to topological
conjugation. \hfill $\square$
\end{thm}

\subsection{Actions with 1-punctured M\"obius band as the quotient orbifold.}
We consider the actions on orientable and non-orientable surfaces separately.
\begin{thm}\label{mb1ori}
There is an action of a cyclic group of order $N$ on a bordered
orientable surface  with $k$ boundary components, with a M\"obius band
having $1$ cone point of order $m$ as the quotient orbifold, if and
only if $k$ divides $N$, $N=2\mathrm{lcm}(m,N/k)$, and either
$t=(m,N/k)$ is odd, or $N/2t$ is even. Furthermore
in such case the  algebraic genus of the surface is equal to $1+(m-1)N/m$ and
there are
$\lceil \varphi(t)/2 \rceil$ conjugacy classes of such actions.
\end{thm}
\begin{thm}\label{mb1nonori}
There is an action of a cyclic group of order $N$ on a bordered
non-orientable surface
with $k$ boundary components, with a M\"obius
band having $1$ cone point of order $m$ as the quotient orbifold, if
and only if $k$ divides $N$, $N=\mathrm{lcm}(m,N/k)$, and for
$t=(m,N/k)$, $N/t$ is odd. Furthermore
in  such case the  algebraic genus of the surface is equal to  $1+(m-1)N/m$
and
the number of topological
conjugacy classes of such actions is $\varphi(t)$ or $\left\lceil
\varphi(t)/2\right\rceil$ if $N$ is even or odd respectively.
\end{thm}

\subsection{Actions with a $2$-punctured disc as the quotient orbifold.}
\begin{thm}\label{d21}  There exists an action of a cyclic group
of order $N$ on a bordered surface $S$  with $k$ boundary components,
having a disc with two cone points of orders $m$ and $n$ as the
quotient orbifold if and only if $S$ is orientable and
\begin{itemize}
\item $N=\mathrm{lcm}(m,n)$,
\item $k$ divides ${t}/{(t,N/t)}$, where $t=(m,n)$,
\item if $N$ is even and $N/t$ is odd, then $k$ is even,
\end{itemize}
In such case the  algebraic genus of  $S$ is equal to   $1+N(1-1/m-1/n)$ and if  $C$ denotes the biggest
divisor of $t/k$ coprime to $Nk/t$, then the number of equivalence
classes of such actions is

\hspace{2mm}
\begin{tabular}{lll}
$\bullet$& $\varphi\big(t/kC\big)\psi(C)$ & \;{if}\; $n\ne m$,\\
$\bullet$& $\varphi\big(n/kC\big)\psi(C)/2 +1$ & \;{if}\; $n=m$ and $n/kC=2^z$, where $z>1$\\
$\bullet$& $\left\lceil\varphi(n/kC)\psi(C)/2\right\rceil$ &
\;{otherwise}.
\end{tabular}
\end{thm}

\subsection{Actions with a 1-punctured annulus as the quotient orbifold.}
First we deal with the actions on non-orientable surfaces.
\begin{thm}\label{ann1-nonori}
There exists an action of a cyclic group $\zz_N$ on a
non-orientable surface $S$  with $k$ boundary components, with an
annulus having one cone point of order $m$ as the quotient orbifold,
if and only if $k$ divides $N$and $N=\mathrm{lcm}(m,N/k)$. Furthermore,
in  such case the  algebraic genus of the surface is equal to $1+N(m-1)/m$ and
there are $\varphi(t)$ different topological types of such action,
where $t=(m,N/k)$.
\end{thm}

%\grom{
The case of orientable $S$ considered in the next theorem 
is much more involved.
It has two parts. The first describes the necessary and sufficient conditions   for  existence of the actions, whereas  the 
the second has quantitative character and  provides  the numbers of equivalence classes of such actions. These numbers are expressed in terms of BSK-maps 
and therefore a reader less familiar with the study of periodic  actions on compact surfaces from a combinatorial point of view should  postpone the reading  of  (ii)-(iv)
until Section 3,  where these maps are introduced.
%}

\begin{thm}\label{ann1}
$(i)$ There exists an action of a cyclic group $\zz_N$ on an orientable
surface $S$ with $k$ boundary components, with an annulus having one
cone point of order $m$ as the quotient orbifold if and only if
either

\noindent
$(1)$ $k$ divides $N$, $N=2\mathrm{lcm}(m,N/k)$ and $N/2$ is odd,
or\\ $(2)$ $m$ divides $ N$ and there exits an integer $n$, $1\le
n<k$, such that:
\begin{itemize}
\item[(a)] $n$ and $k-n$ divide $m$;
\item[(b)] $N/m$, $n$ and $k-n$ are pairwise relatively prime;
\item[(c)] if $N$ is even then one of $N/m$, $n$, $k-n$ is even.
\end{itemize}
In   such case the  algebraic genus of the surface is equal to  $1+N(m-1)/m$.

\noindent
$(ii)$ Suppose that $N, m, k$ satisfy $(1)$ and $t=(m,N/k)$. Then there are
$\varphi(t)$ equivalence classes of BSK-maps $\theta^1\colon\Lambda\to\Z_N$ such that
$\theta^1(c_1)\ne\theta^1(c_2)$.

\smallskip

\noindent
$(iii)$ Suppose that $N, m, k$ satisfy $(2)$. Then the number of equivalence classes of BSK-maps $\theta^2\colon\Lambda\to\Z_N$ such that
$\theta^2(c_1)=\theta^2(c_2)=0$, and $\theta^2(e_1)$ and $\theta^2(e_2)$ have orders $N/n$ and $N/(k-n)$ is

\hspace{2mm}\begin{tabular}{lll}
$\bullet$ & $\varphi\big({m}/{Cn(k-n)}\big)\psi(C)$ & \;{if}\; $k\ne 2$,\\
$\bullet$ & $\varphi\big({m}/{C}\big)\psi(C)/2+1$ & \;{if}\; $k=2$ and $m/C=2^z$, where $z>1$,\\
$\bullet$ &
$\left\lceil\varphi\big({m}/{C}\big)\psi(C)/2\right\rceil$ &
\;{otherwise},
\end{tabular}

\noindent
where $C$ is the biggest divisor of
${m}/{n(k-n)}$ coprime to ${N}n(k-n)/{m}$.

\smallskip
\noindent
$(iv)$ Every BSK-map corresponding to a $\zz_N$-action on $S$ is equivalent either to some $\theta^1$ from the assertion (ii), or to some $\theta^2$ from the assertion (iii).
\end{thm}

\subsection{Actions with a 3-punctured disc as the quotient orbifold.} The orders of the three cone points are either $2,2,m$, $m\ge 2$, or $2,3,m$, where $m\in\{3,4,5\}$. We consider these two cases separately.
\begin{thm}\label{d31}
There is an action of a cyclic group of order $N$ on a bordered
surface $S$ with a disc having $3$ cone points of orders $2,2,m$ as
the quotient orbifold if and only if $N=\mathrm{lcm}(2,m)$.
In such case $S$ is orientable, it has $N/m$ boundary components,
{its genus} is equal to
 $g=1+(m-2)N/2m$ and
such an action is unique up to topological conjugation.
\end{thm}

\begin{thm}\label{d32}
There is an action of a cyclic group of order $N$ on a bordered
surface $S$ with a disc having $3$ cone points of orders $2,3,m$,
where $m\in\{3,4,5\}$, as the quotient orbifold if and only if
$N=\mathrm{lcm}(2,3,m)$ and $S$ is an orientable surface of topological genus
{$g$} with $k$ boundary components, where
\begin{itemize}
\item if $m=3$ then $(g,k)=(3,1)$ or $(2,3)$,
\item if $m=4$ then $(g,k)=(6,1)$,
\item if $m=5$ then $(g,k)=(15,1)$.
\end{itemize}
In each case the action is unique up to topological conjugation.
\end{thm}

\subsection{Actions with a 2-punctured disc with two corners as the quotient orbifold.} The orders of the cone points are either $2,m$, $m\ge 2$, or $3,m$, where $m\in\{3,4,5\}$. We consider these two cases separately.
\begin{thm}\label{d221}
There is an action of a cyclic group of order $N$ on a bordered
surface $S$ with a disc having $2$ cone points of orders $2,m$ and
$2$ corners as the quotient orbifold if and only if
$N=\mathrm{lcm}(2,m)$.
In such case  $S$ is  non-orientable, it has  $N/2$ boundary
components, {its genus is} equal to  $g=2+(m-2)N/2m$, and
 such an action is unique up to
topological conjugation.
\end{thm}

\begin{thm}\label{d222}
There is an action of a cyclic group of order $N$ on a bordered
surface $S$ with a disc having $2$ cone points of orders $3,m$,
where $m\in\{3,4,5\}$, and $2$ corners as the quotient orbifold if
and only if $N=\mathrm{lcm}(2,3,m)$, $S$ has $N/2$ boundary
components, and
\begin{itemize}
\item if $m=3$ then $S$ is orientable of genus $2$;
\item if $m=4$ then $S$ is non-orientable of genus $7$;
\item if $m=5$ then $S$ is orientable of genus $8$.
\end{itemize}
Furthermore, for $m=3$ there are two different topological types of
such action, and for $m=4,5$ the action is unique up to topological
conjugation.
\end{thm}

\section{Preliminaries}\label{sec:preli}
In principle, we  use a combinatorial approach, based on Riemann
unformization theorem for compact Riemann surfaces, its
generalization for non-orientable or bordered surfaces with dianalytic
structures of Klein surfaces, good knowledge of discrete group of
isometries of the hyperbolic plane and some elementary covering
theory. For the reader's convenience, we review the terminology of \cite{BEGG} used in this paper.

\subsection{Hurwitz-Nielsen geometrization and its generalizations.}\label{geometrization}
Let $G$ be a finite group of orientation-preserving
self-homeo\-mor\-phisms of a closed orientable surface $S_g$ of genus $g$, $g\ge 2$.

\smallskip
By \cite{Hur} and \cite{Niel}, there exists a structure of a Riemann
surface on $S_g$, with respect to which
the elements of $G$ act as
conformal automorphisms. This result was generalized to the
case of actions containing orientation-reversing self-homeomorphisms
and to closed non-orientable surfaces by B. Kerejarto \cite{Ker}, and for bordered surfaces in a more recent monograph of Alling and Greanleaf
\cite{AG}, who introduced the concept of a Klein surface. Thus, although the paper concerns topological classification  of topological actions,   we assume, whenever necessary, that a surface has such a structure of a bordered Klein surface, and the elements of $G$ act on it as dianalytic automorphims. This assumption allows for effective conformally-algebraic methods  described in the following subsections.

\subsection{Non-euclidean crystallographic groups.}
By a non-euclidean crystallographic group (NEC-group in short) we
mean a discrete and cocompact subgroup of the group ${\mathcal G}$ of
all isometries of the hyperbolic plane ${\mathcal
 H}$.
The algebraic structure of such a group $\Lambda$ is encoded in its signature:
\begin{equation}\label{sign}
s({\Lambda}) = (g;\pm;[m_1, \ldots , m_r];\{(n_{11}, \ldots , n_{1s_1}), \ldots , (n_{k1}, \ldots , n_{ks_k})\}),
\end{equation}
where the brackets $(n_{i1},\ldots ,n_{is_i})$ are called {\it the
period cycles}, the integers $n_{ij}$ are the {\it link periods},
 $m_i$ {\it proper periods }
and finally $g$ the {\it orbit genus} of ${\Lambda}$. A group $\Lambda$ with signature (\ref{sign}) has
the presentation with the following generators

\medskip
\hspace{1mm}
\begin{tabular}{ll}
$x_i$ & \;\;\;\;\;\;\;\;\;\;\;\;\;\;\;\;\;\;\;\;\;\;\;\;\;\;\;\;\;\;\;\;\;\;\;\;\;\;\;\;\;\;\;\;\;\; \; for $1\le i\le r$,\\
$c_{ij}$, {$e_i$} & \;\;\;\;\;\;\;\;\;\;\;\;\;\;\;\;\;\;\;\;\;\;\;\;\;\;\;\;\;\;\;\;\;\;\;\;\;\;\;\;\;\;\;\;\;\; \; for $1\le i\le k, 0\leq j \leq s_i$,\\
$a_i$, $b_i$ & \;\;\;\;\;\;\;\;\;\;\;\;\;\;\;\;\;\;\;\;\;\;\;\;\;\;\;\;\;\;\;\;\;\;\;\;\;\;\;\;\;\;\;\;\;\; \; for $1\le i\le g$ if the sign is $+$,\\
$d_i$ & \;\;\;\;\;\;\;\;\;\;\;\;\;\;\;\;\;\;\;\;\;\;\;\;\;\;\;\;\;\;\;\;\;\;\;\;\;\;\;\;\;\;\;\;\;\; \; for $1\le i\le g$ if the sign is $-$,
\end{tabular}

\medskip
\noindent
subject to the relations

\medskip
\hspace{1mm}
\begin{tabular}{ll}
$x_i^{m_i}=1$ &\;\,for $1\le i\le r$,\\
$c_{ij}^2=(c_{ij-1}c_{ij})^{n_{ij}}=1$ &\;\,for $1\le i\le k, 0\leq j \leq s_i$,\\
$c_{is_i}=e_ic_{i0}e_i^{-1}$ &\;\,for $1\leq i\leq k$,\\
$ x_1\cdots x_r e_1\cdots e_k[a_1,b_1]\cdots[a_g,b_g]=1$ &\;\,if the sign is $+$,\\
$x_1\cdots x_r e_1\cdots e_kd_1^2\cdots d_g^2 =1$ &\;\,if the sign is $-$,\\
\end{tabular}

\noindent
where $[x,y]=xyx^{-1}y^{-1}$.
Elements of any system of generators satisfying the above relation
will be called {\it canonical generators}. The elements $x_i$ are
elliptic transformations, $a_i, b_i$ hyperbolic translations, $d_i$
glide reflections and $c_{ij}$ hyperbolic reflections. Reflections
$c_{ij-1}$ and $c_{ij}$ are called consecutive. It is essential for
applications that every element of finite order in
$\Lambda$ is conjugate either to a canonical reflection, or to a
power of some canonical elliptic element $x_i$, or else to a power of
the product of two consecutive canonical reflections.

{The orbit space ${\mathcal H}/\Lambda$ is a hyperbolic orbifold, with underlying surface of topological genus $g$ with $k$ boundary components, and it is orientable if the sign is $+$ and non-orientable otherwise.
The image in ${\mathcal H}/\Lambda$ of the fixed point of the canonical elliptic generator $x_i$ is called {\it cone
point} of order $m_i$, whereas the image of the fixed point of the product of two consecutive canonical reflections $c_{ij-1}c_{ij}$ is called {\it
corner point} of order $n_{ij}$.}

\smallskip
Now, an abstract group with such presentation can be realized as an
NEC-group $\Lambda$ if and only if the value
\begin{equation}\label{area}\varepsilon g +k -2
+\sum_{i=1}^r{\left(1-{1\over{m_i}}\right)} +
 {1\over2} \sum_{i=1}^k \sum_{j=1}^{s_i}
 {\left(1-{1\over{n_{ij}}}\right)}
\end{equation}
is positive,
where $\varepsilon = 2$  if the sign is $+$, or $\varepsilon=1$ otherwise. This value turns out to be the normalized
 hyperbolic area
 $\mu (\Lambda)$ of an arbitrary
 fundamental region for such a group, and we have the following
 Hurwitz-Riemann formula
\begin{equation}\label{HR}
[\Lambda : \Lambda'] = \dfrac{\mu(\Lambda')}{ \mu(\Lambda)}
\end{equation}
for a subgroup $\Lambda'$ of finite index in an NEC-group $\Lambda$.

\smallskip

Finally, NEC-groups without orienta\-tion-re\-ver\-sing elements are
Fuchsian groups. They have signatures $(g;+;[m_1,\ldots ,
m_r];\{-\})$ usually abbreviated as $(g;m_1,\ldots , m_r)$.
Given an NEC-group $\Lambda$ containing orientation-reversing elements,
its subgroup
$\Lambda^+$   consisting of the orientation-preserving
elements is called the {\it canonical Fuchsian subgroup of}
$\Lambda$, and by \cite{S}, for  $\Lambda$ with signature (\ref{sign}), $\Lambda^+$ has signature
\begin{equation}\label{can-fuchs}(\varepsilon g +k-1;m_1,m_1,
\ldots ,m_r,m_r,n_{11},\ldots , n_{ks_k}).
\end{equation}
A torsion free Fuchsian group $\Gamma$ is called a {\it surface group} and it
has signature $(g;-)$.

We will also use other results concerning relationship between the signatures
of an NEC-group and its finite index subgroup proved in Chapter 2 of \cite{BEGG}.

\subsection{Bordered Riemann surfaces and their groups of automorphisms}
By the Riemann uniformization theorem, every closed Riemann surface $S$ of genus $g \geq 2$ can be identified with the orbit space ${\mathcal H}/\Gamma$ of the hyperbolic
plane with respect to an action of a Fuchsian group $\Gamma$ isomorphic to the fundamental group of $S$.
A Klein surface is a compact bordered topological surface equipped with a
dianalitic structure - historically it is also called bordered
Riemann surface. For a given Klein surface $S$, Alling and Greenleaf
\cite{AG} constructed certain canonical double cover $S^+$ being a
Riemann surface, {such that $S$ is the quotient of $S^+$ by an action of an anti-holomorphic involution with fixed points}.
The {\it algebraic genus} $p=p(S)$  of $S$ is defined as the
genus of $S^+$ and it follows from the construction %of Alling and Greenleaf
 that {$p$} coincides with the rank of the fundamental group of $S$,
and so
for a surface of topological genus $g$ having $k$
boundary components it is equal to {$p=\varepsilon g+k-1$},
where $\varepsilon =2$ if $S$ is orientable and $\varepsilon =1$ otherwise.
It is well known (see \cite{BEGG} for example) that any compact
Klein surface $S$ of algebraic genus $p\geq 2$ can be represented as
${\mathcal H}/\Gamma$ for some NEC-group $\Gamma$. If $S$ has
topological genus $g$ and $k$ boundary components, then $\Gamma$ can
be chosen to be a {\it bordered surface group}, i.e, an NEC-group with
the signature
\begin{equation}\label{surface}
(g;\pm;[\;];\{(\,),\stackrel{k}{\ldots},(\,)\}),\end{equation}
whose only elements of finite order are reflections.
It has the presentation
$$
\langle a_1, b_1, \ldots, a_g, b_g, e_1, \ldots, e_k,c_1, \ldots, c_k \; |\; c_i^2, [e_i,c_i], e_1 \ldots e_k[a_1,b_1] \ldots [a_g,b_g] \rangle
$$
if the sign is $+$, or
$$
\langle d_1, \ldots, d_g, e_1, \ldots, e_k,c_1, \ldots, c_k \; | \;
c_i^2, [e_i,c_i], e_1 \ldots e_k d_1^2 \ldots d_g^2 \rangle
$$
otherwise. Finally, a finite group $G$ is a group of automorphisms
of $S={\mathcal H}/\Gamma$ if and only if $G\cong \Lambda /\Gamma$
for some NEC-group $\Lambda$. A convenient way of defining an
action of a group $G$ on a bordered surface $S$ is by means of an epimorphism
$\theta: \Lambda \to G$ whose kernel is a bordered surface group.
In such a case $S= {\mathcal H}/\Gamma$,
where $\Gamma= \ker \theta$. We shall refer to such an epimorphism as
to a bordered-surface-kernel epimorphism (BSK in short) or
smooth-epimorphism.

\smallskip
Two  actions of $G$ on $S$ are topologically conjugate
(by a homeomor\-phism of $S$) if and only if the associated smooth
epimorphisms are equivalent in the sense of the next definition (see
\cite[Proposition 2.2]{BCCS}). We say that two smooth epimorphisms
$\theta_i\colon\Lambda \to G$, $i=1,2$, are {\it equivalent} if and
only if there exist automorphisms $\phi\colon\Lambda\to\Lambda$ and
$\varphi\colon G\to G$ such that the following diagram is
commutative.
\begin{equation}\label{def top equiv}
\xymatrix{
\,\Lambda\ar[r]^\phi&{\,\Lambda}\\\;G\ar[r]^\varphi\ar@{<-}[u]^{\theta_1}&G\!\ar@{<-}[u]_{\theta_2} }.
\end{equation}

\subsection{Some elementary algebra} For integers $a, b
$ we denote by $(a, b)$ their greatest common
divisor and we use  additive notation for cyclic groups
$\mathbb{Z}_N=\mathbb{Z}/N\mathbb{Z}$ throughout the whole paper.
Furthermore, by abuse of language, we  write $a\in \mathbb{Z}_N
$ for a non-negative integer $a <N$.
%Throughout the whole paper $\varphi$ will  denote the Euler
%totient function and
To avoid unnecessary parentheses, we denote expressions of the form $a/(bc)$ simply as $a/bc$.

We will need the following version of the classical Chinese
Remainder Theorem
\begin{lemma}\label{CRT}
Given integers $a,b$, the system of congruences
$$
\begin{cases}
x \equiv a \tmod m\\
x \equiv b \tmod n
\end{cases}
$$
has a solution if and only if $a\equiv b \tmod t$, where $t=(m,n)$
and this solution is unique up to ${\rm lcm} (m,n)$. \hfill $\square$
\end{lemma}
%\begin{pf}If $x$ is a solution then $m$ divides $x-a$ and $n$ divides $x-b$ which prove that {$t$} divides $a-b$. Conversely, {if} $a \equiv b {\tmod t}$ then {$a-b = \alpha m + \beta n$} for some  integers $\alpha, \beta$ and $a - \alpha m= b+ \beta n$ is a solution $x$ we are looking for.  \end{pf}

The following useful  result can be proved using Dirichlet's
theorem on arithmetic progression (see \cite{GSZ} for a more elementary, direct, argument).

\begin{lemma}\label{DT}
Given an integer $N$ and its divisor $n$, the reduction map
$\Z_N^\ast \to \Z_n^\ast$ is a group epimorphism.
\end{lemma}
\begin{pf}Let $a \in \zz_n^\ast$, then $(a,n)=1$ and so by {Dirichlet}
theorem on arithmetic progression there exists infinitely many
primes $A$ of the form $a + bn$ and so $A\in \zz_N^\ast$ and its
reduction modulo $n$ is equal to $a$.
\end{pf}

We will also  need
\begin{lemma}[Harvey, \cite{Har}]\label{lem-Harv}
The group $\zz_N$ is generated by three elements $a$, $b$, $c$ of
orders $m$, $n$, $l$ and such that $a+b+c=0$ if and only if
\begin{itemize}
\item[(i)] $N=\mathrm{lcm}(m,n)=\mathrm{lcm}(m,l)=\mathrm{lcm}(n,l)$, and
\item[(ii)] if $N$ is even, then exactly one of the numbers $N/m$, $N/n$, $N/l$ is even. \hfill $\square$
\end{itemize}
\end{lemma}

The condition (i) of Lemma \ref{lem-Harv} is equivalent to existence
of pairwise relatively prime integers $A, A_1, A_2, A_3$ for which
\begin{equation}\label{Maclachlan}
m=AA_2A_3,\ n=AA_1A_3,\ l=AA_1A_2,\ N=AA_1A_2A_3.
\end{equation}
The condition (ii) of Lemma \ref{lem-Harv} is that one of the
numbers $A_1$, $A_2$, $A_3$ is even if $N$ is even. The quadruple
$(A,A_1,A_2,A_3)$ is called {\it Maclachlan decomposition} of the
triple $(m,n,l)$
after R. Hidalgo \cite{Hidalgo}.

\section{Periodic self-homeomor\-phisms of compact bordered surfaces of big periods}\label{sec:types}
From (3.1.0.1) and (3.1.0.2) on page 61 in \cite{BEGG} we immediately obtain the following result.

\begin{lemma}\label{aut rep}
There exists a structure of a bordered Klein surface
$S=S_{g,\pm}^{k}$ of topological genus $h$, with $k$ boundary
components and orientable or not according to the sign being plus or
minus having a dianalytic automorphism $\varphi$ of order $N$ if and
only if $\zz_N \cong \Lambda / \Gamma$, where $\Gamma $ and
$\Lambda$ have signatures respectively
\begin{equation}\label{sygn pm}
(g;\pm ;[\;];\{(\,)\stackrel{k}{\ldots}\,,(\,)\}) \;\; {\rm and} \;\;
(g'; \pm ;[m_1, \ldots, m_r];\{C_1,\ldots,C_{k'}\}),
\end{equation}\\[-4mm]
where each cycle $C_i$ is either empty or consists of an even number
of periods equal to $2$. Furthermore, non-empty cycles do not
appear for odd $N$. \hfill $\square$
\end{lemma}

For the rest of this section we assume that $\Lambda$ and
$\Gamma$ are as in Lemma \ref{aut rep} and we denote by
$\theta\colon\Lambda\to\Lambda/\Gamma$ the canonical projection,
where $\Lambda/\Gamma$ will be identified with $\zz_N$. Recall that
$\theta$ is called BSK-epimorphism. In order to state the next
lemma, we need one definition and some notation. We define a {\it
non-orientable word} to be a word $w$ in the canonical generators of
$\Lambda$ and their inverses, such that $w$ defines an
orientation-reversing isometry of $\mathcal{H}$ and the reflections
$c_{ij}$ of $\Lambda$ which belong to $\Gamma$ do not appear in
$w$. For $i=1,\dots,k'$ let $s_i$ be the length of the cycle $C_i$
in the signature of $\Lambda$. Let $l_i$ denote the order of
$\theta(e_i)$ in $\zz_N$ and
$$
t_i=\begin{cases}
0&\textrm{if\ $s_i=0$\ and\ $c_{i0}\notin\Gamma$}\\
N/l_i&\textrm{if\ $s_i=0$\ and\ $c_{i0}\in\Gamma$}\\
s_iN/4&\textrm{if\ $s_i>0$.}
\end{cases}
$$
The following lemma can be deduced from \cite{BEGG}.
\begin{lemma}\label{sign k}
Let $S$ be the surface $\mathcal{H}/\Gamma$. Then
\begin{itemize}
\item[(a)] $S$ is non-orientable if and only if $\Gamma$ contains a non-orientable word;
\item[(b)] the number of boundary components of $S$ is $k=t_1+\dots+t_{k'}$.
\end{itemize}
\end{lemma}
\begin{pf}Indeed (a) follows from Theorems 2.1.2 and 2.1.3 in \cite{BEGG}.
Now, with the above notation, each empty period cycle of $\Lambda$
whose corresponding canonical reflection belongs to $\Gamma$, produces
$N/l_i$ empty period cycles in $\Gamma$
by Theorem 2.3.3 in \cite{BEGG} (see also theorems 2.4.2 and 2.4.4 therein).
Next, let $c_0, c_1, \ldots, c_{2s}$ be
a cycle of canonical reflections of $\Lambda$ corresponding to a
nonempty period cycle $(2, \stackrel{2s}{\ldots}\,,2)$. Then
$\theta(c_i)=0$ or $N/2$. Observe however that two consecutive
canonical reflections $c_{i-1},c_i$ have different images, since
otherwise $c_{i-1}c_i$ would be an orientation preserving torsion
element of $\Gamma$. So this cycle of reflections is mapped {either} on
$0,N/2,0,N/2, \ldots, 0$ or on $N/2,0,N/2, \ldots, 0,N/2$.
In the former case each $c_i$ for even $i$ produces
in $\Gamma$ $N/2$ empty period cycles, while in the latter case the
same is true for every odd $i$, and so each nonempty period cycle of length $2s$ produces in $\Gamma$ $s(N/2)$ empty period
cycles in virtue of Theorem 2.3.2 in \cite{BEGG} (see also Theorem 2.4.4 therein).
\end{pf}

Observe that we can determine the topological type of the surface
$\mathcal{H}/\Gamma$ by using Lemma \ref{sign k} together with the
Hurwitz-Riemann formula.

\begin{lemma}\label{big o-r}
If $\Gamma$ is a bordered surface group of algebraic genus $p$ and
$N>p-1$ then $\Lambda$ has one of the following signatures:

\medskip
\begin{tabular}{rlrl}
$(1)$&$(0;+;[\;];\{(2,2,2,2,2,2)\})$, & $(2)$&$(0;+;[\;];\{(\,),(2,2)\})$,\\
$(3)$&$(1;-;[\;];\{(2,2)\})$, & $(4)$&$(0;+;[m];\{(2,2)\})$,\\
$(5)$&$(0;+;[m];\{(2,2,2,2)\})$, & $(6)$&$(1;-;[m];\{(\,)\})$,\\
$(7)$&$(0;+;[m,n];\{(\,)\})$, & $(8)$&$(0;+;[m];\{(\,),(\,)\})$,\\
$(9a)$&$(0;+;[2,3,m];\{(\,)\})$, $m=3,4,5,\;\;\;\;\;\;\;\;\;\;\;\;\;\;\;\;\;\;$ & $(9b)$&$(0;+;[2,2,m];\{(\,)\})$,\\
$(10a)$&$(0;+;[3,m];\{(2,2)\})$, $m=3,4,5$, &
$(10b)$&$(0;+;[2,m];\{(2,2)\}).$
\end{tabular}
\end{lemma}

\begin{pf}By the Hurwitz-Riemann formula, $N>p-1$ is equivalent to
$\mu(\Lambda)<1$. For $\Lambda$ as in Lemma \ref{aut rep}, we have
$$\mu(\Lambda)=\varepsilon g'+k'-2+\sum_{i=1}^r\left(1-\frac{1}{m_i}\right)+\frac{s}{4},$$
where $s$ is the sum of lengths of non-empty period cycles. Observe
that $s$ is even and $k'>0$. We have $-1\le \varepsilon g'+k'-2\le 0$.

Suppose that $\varepsilon g'+k'-2=0$. Then, since $(1-1/m_i)\ge
1/2$, $0\le r\le 1$. If $r=0$, then $s=2$ and $\Lambda$ has
signature (2) or (3). If $r=1$, then $s=0$ and $\Lambda$ has
signature (6) or (8).

Suppose that $\varepsilon g'+k'-2=-1$, hence $(g',k')=(0,1)$. Then
$0\le r\le 3$. If $r=0$, then $s=6$ and $\Lambda$ has the signature
(1). If $r=1$, then $s=2$ or $4$ and $\Lambda$ has signature (4) or
(5). If $r=2$, then $s=0$ or $2$. In the former case $\Lambda$ has
signature (7). In the latter case $1/m_1+1/m_2>1/2$ and $\Lambda$
has signature (10a) or (10b). Finally, if $r=3$ then $s=0$,
$1/m_1+1/m_2+1/m_3>1$ and $\Lambda$ has signature (9a) or (9b).
\end{pf}

We close this section by a technical but simple lemma, which will be
very useful in the next section.
\begin{lemma}\label{cycle perm}
Suppose that $C_i$ is a non-empty cycle in the signature of
$\Lambda$. Then for some $\phi\in\mathrm{Aut}(\Lambda)$,
$\theta\circ\phi$ maps the corresponding reflections
$(c_{i0},c_{i1},\dots,c_{is_i})$ on $(N/2,0,\dots,N/2,0,N/2)$.
\end{lemma}
\begin{pf}
 By Lemma \ref{aut rep},
all the periods in $C_i$ are equal to $2$ and by the proof of Lemma
\ref{sign k}, consecutive canonical reflections
$c_{i0},c_{i1},\dots,c_{is_i}$ are mapped either on
$N/2,0,\dots,N/2,0,N/2$, or on $0,N/2,\dots,0,N/2,0$. In the former
case we take $\phi$ to be the identity, while in the latter case we
define $\phi$ by $\phi(c_{ij})=c_{ij-1}$ for $j=1,\dots,s_i$,
$\phi(c_{i0})=e_i^{-1}c_{is_i-1}e_i$, and the identity on the
remaining generators of $\Lambda$.
\end{pf}

%%%%%%%%%%%%%%%%%%%%%%%%%%%%%%%%%%

%%%%%%%%%%%%%%%%%%%%%%%%%%%%%%%%%%

\section{Automorphisms of NEC-groups vs mapping class groups.}\label{ss:MCG}
From  diagram (\ref{def top equiv}) in  Section \ref{sec:preli} we
see that for a topological classification of group actions via smooth epimorphisms we need to know how to calculate automorphisms groups of
NEC-groups $\Lambda$. As we shall see, we need to know these
automorphisms up to conjugation, which means that we actually need
the groups ${\rm Out}(\Lambda)$ of outer automorphisms of
$\Lambda$'s. From the previous section we see that in this paper we  need them only for three signatures
$(1;-;[m];\{(\,)\})$, $(0;+;[m];\{(\,),(\,)\})$, $(0;+;[m,n];\{(\,)\})$,
and  the outer automorphism groups for these NEC-groups were found in
 \cite[\S 4]{BCCS} by using a connection between ${\rm
Out}(\Lambda)$ and the mapping class group of the orbifold
${\mathcal H}/\Lambda$.  For reader's convenience we review these results and their proofs (illustrated with figures and easier to follow then the proofs in \cite{BCCS}).   For an NEC-group $\Lambda$, let
$\mathrm{Mod}({\mathcal H}/\Lambda)$ be the group of isotopy classes
of homeomorphisms over ${\mathcal H}/\Lambda$ which map a cone point
to a cone point of the same order, and analogously for the corner
points, and $\mathrm{PMod}({\mathcal H}/\Lambda)$ be the group of
isotopy classes of homemorphisms over ${\mathcal H}/\Lambda$ which
fix the cone points and the corner points. For two elements
$\phi_1$, $\phi_2$ of $\mathrm{Mod}({\mathcal H}/\Lambda)$, $\phi_1
\phi_2$ means applying $\phi_2$ first and then applying $\phi_1$.
Let ${\rm Out}_0(\Lambda)$ be the subgroup of ${\rm Out}(\Lambda)$
which acts trivially on the set of conjugacy classes of the
stabilisers of the fixed points of elliptic elements of $\Lambda$.
Observe that these conjugacy classes are in one to one correspondence with the integers
$m_1,\dots,m_r,n_{11},\dots,n_{ks_k}$, and hence ${\rm Out}_0(\Lambda)$  is a subgroup of finite index of ${\rm Out}(\Lambda)$.

\begin{lemma}\cite[Corollary 4.4]{BCCS}\label{L:PMod_Out0}
If $\mathrm{PMod}({\mathcal H}/\Lambda)$ has finite order $n$, then ${\rm Out}_0(\Lambda)$ has order at most $n$.
\hfill $\square$
\end{lemma}

\begin{figure}[hbtp]
\includegraphics[height=5cm]{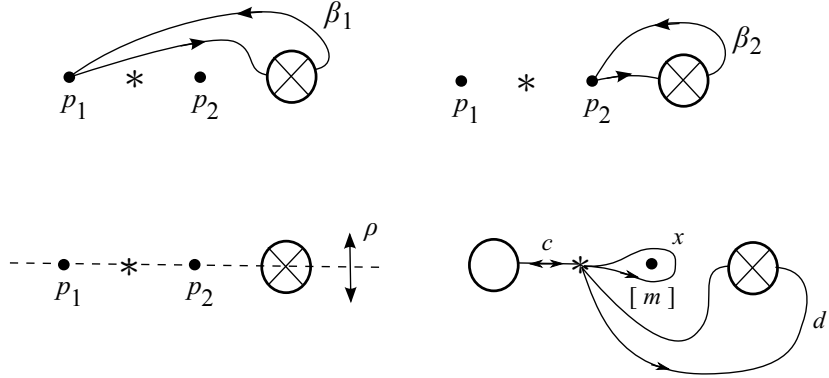}
\caption{$\otimes$ indicates the place to attach a M\"{o}bius band.
The loops $x$, $c$ and $d$ represent the generators of $\Lambda$
with signature $(1;-;[m];\{(\,)\})$ or the orbifold fundamental group
of ${\mathcal H}/\Lambda$ whose base point is *. }
\label{fig:N12}
\end{figure}

In order to obtain  presentations of these groups, we review  mapping
class groups of two elementary surfaces.
Let $S_{0,3}$ be the sphere with three marked points $p_1$, $p_2$ and $p_3$,
and $\mathrm{PMod}(S_{0,3})$ be the group of isotopy classes of orientation preserving
diffeomorphisms over the sphere preserving each of these three points.
It is well-known that $\mathrm{PMod}(S_{0,3})$ is trivial (see, for example,
the proof of Proposition 2.3 in \cite{FM}).
Let $N_{1,2}$ be the real projective plane with two marked points $p_1$ and $p_2$,
and $\mathrm{PMod}(N_{1,2})$ be the group of isotopy classes of
diffeomorphisms over the real projective plane preserving each of these two points.
Let $\beta_1$ and $\beta_2$ be the oriented circles shown in Figure \ref{fig:N12},
and $\nu_i$ be the element of $\mathrm{PMod}(N_{1,2})$ obtained by sliding
$p_i$ once along $\beta_i$ ($i = 1,2$).
Korkmaz \cite[Corollary 4.6]{K} showed that $\mathrm{PMod}(N_{1,2})$ is generated by $\nu_1$
and $\nu_2$, and $\nu_1^2 = \nu_2^2 = (\nu_2 \nu_1)^2 =1$.
Let $\rho$ be the reflection indicated in Figure \ref{fig:N12}.
By investigating the action on the fundamental group, we can see
$\rho = \nu_1 \nu_2$.
Therefore, we see that $\mathrm{PMod}(N_{1,2})$ is generated by $\nu_1$ and $\rho$,
and $\nu_1^2 = \rho^2 = (\rho \nu_1)^2 =1$.

\begin{lemma}\cite[Proposition 4.12]{BCCS}\label{mb1-lem}
Let $\Lambda$ be an NEC-group with signature $(1;-;[m];\{(\,)\})$
and canonical generators $x$, $d$, $c$, satisfying the relations
$x^m=c^2=1, d^2xc=cd^2x$.
Then $\mathrm{Out}(\Lambda)$ is isomorphic to the Klein four-group and
is generated by classes of automorphisms $\gamma$, $\delta$ defined by
$$
\;\;\;\;\;\;\;\;\;\;\gamma\colon\begin{cases}x\mapsto x^{-1}\\ d\mapsto x^{-1}d^{-1}x\\c\mapsto c\end{cases}\qquad
\delta\colon\begin{cases}x\mapsto x\\ d\mapsto (dx)^{-1}\\c\mapsto (dx)^{-1}c(dx)\end{cases}
$$
\end{lemma}
\begin{pf}Under the correspondence $p_1$ to the boundary, and $p_2$ to the cone point,
$\mathrm{PMod}(N_{1,2})$ is isomorphic to
$\mathrm{Mod}({\mathcal H}/\Lambda) = \mathrm{PMod}({\mathcal H}/\Lambda)$.
The action of $\rho$ on $\Lambda$ is $\gamma$ and that of $\nu_1$ is $\delta$
and these actions are of order 2 and not inner automorphisms of $\Lambda$.
By Lemma \ref{L:PMod_Out0}, the order of ${\rm Out}_0(\Lambda) = {\rm Out}(\Lambda)$
is at most 4.
Therefore, we see that ${\rm Out}(\Lambda)$ is the Klein four-group generated
by $\gamma$ and $\delta$.
\end{pf}

\begin{lemma}\label{ann1-lem}
Let $\Lambda$ be an NEC-group with signature $(0;+;[m];\{(\,),(\,)\})$
with canonical generators $x, e, c_1, c_2$ satisfying the following
defining relations: $x^m=c_1^2=c_2^2=1,\quad ec_1=c_1e,\quad
xec_2=c_2xe.$ Then $\mathrm{Out}(\Lambda)$ is isomorphic to the
Klein four-group and is generated by classes of automorphisms
$\alpha$, $\beta$ defined by
$$\alpha\colon\begin{cases}x\mapsto e^{-1}x^{-1}e\\ e\mapsto e^{-1}\\c_1\mapsto c_1\\c_2\mapsto c_2\end{cases}\qquad
\beta\colon\begin{cases}x\mapsto e^{-1}xe\\ e\mapsto (xe)^{-1}\\c_1\mapsto c_2\\c_2\mapsto c_1\end{cases}
$$
\end{lemma}
\begin{figure}[hbtp]
\includegraphics[height=6.5cm]{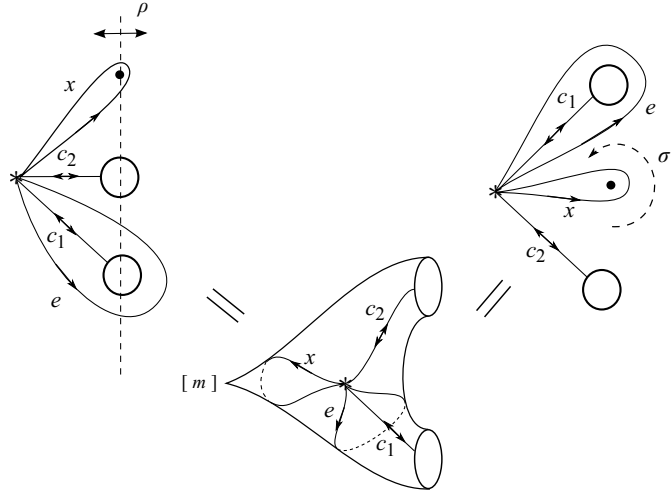}
\caption{
The loops $x$, $c_1$, $c_2$ and $e$ represent the generators of $\Lambda$
with signature $(0;+;[m];\{(\,),(\,)\})$ or the orbifold fundamental group
of ${\mathcal H}/\Lambda$ whose base point is *. }
\label{fig:S21}
\end{figure}
\begin{pf}Let $f$ be a homeomorphism over $\mathcal{H}/\Lambda$ fixing boundaries and
a {cone point},
then we can regard $f$ as an element of $\mathrm{PMod}(S_{0,3})=1$.
Therefore, every element of
$\mathrm{PMod}(\mathcal{H}/\Lambda) = \mathrm{Mod}(\mathcal{H}/\Lambda)$
is determined by its action on the boundary of $\mathcal{H}/\Lambda$.
Let $\rho$ be the reflection about the axis shown in Figure \ref{fig:S21}, and
$\sigma$ be the $\pi$-rotation about the cone point as shown in Figure \ref{fig:S21}.
$\mathrm{PMod}(\mathcal{H}/\Lambda)$ is generated by $\rho$ and $\sigma$, and
its defining relations are $\rho^2 = \sigma^2 = (\rho \sigma)^2 = 1$.
The action of $\rho$ on $\Lambda$ is $\alpha$ and that of $\sigma$ is
$\beta$ and these actions are not inner automorphisms of $\Lambda$.
By Lemma \ref{L:PMod_Out0}, the order of
${\rm Out}_0(\Lambda) = {\rm Out}(\Lambda)$ is at most $4$.
Therefore, we see that ${\rm Out}(\Lambda)$ is the Klein four-group
generated by $\alpha$ and $\beta$.
\end{pf}

\begin{lemma}\label{d2-lem}
Let $\Lambda$ be an NEC-group with a signature $(0;+;[m,n];\{(\,)\})$ and
generators $x_1$, $x_2$, $c$, satisfying the following defining relations:
$x_1^{m}=x_2^{n}=c^2=1,\quad x_1x_2c=cx_1x_2$.
Then if $m\ne n$ then
$\mathrm{Out}(\Lambda)$ has order $2$ and is generated by the class of automorphism $\alpha$ while if $m=n$
and the Klein four-group generated by $\alpha, \beta$ in the other case, where $$
\;\;\;\;\;\;\;\;\;\;\alpha\colon\begin{cases}x_1\mapsto x^{-1}_1\\ x_2\mapsto x_1x_2^{-1}x_1^{-1}\\c\mapsto c\end{cases}\qquad
\beta\colon\begin{cases}x_1\mapsto x_2\\ x_2\mapsto x_2^{-1}x_1x_2\\c\mapsto c\end{cases}
$$
\end{lemma}
\begin{figure}[hbtp]
\includegraphics[height=5cm]{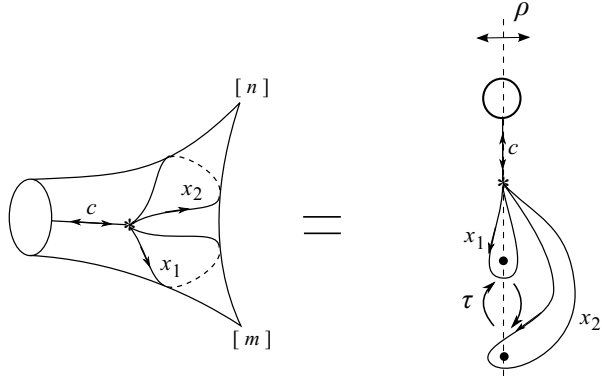}
\caption{
The loops $x_1$, $x_2$, $c$ are the generators of $\Lambda$
with signature $(0;+;[m,n];\{(\,)\})$ or the orbifold fundamental group
of ${\mathcal H}/\Lambda$ whose base point is *. }
\label{fig:S12}
\end{figure}
\begin{pf}Let $f$ be a homeomorphism over $\mathcal{H}/\Lambda$ fixing boundaries and a corner,
then we can regard $f$ as an element of $\mathrm{PMod}(S_{0,3})=1$.
Therefore, every element of $\mathrm{PMod}(\mathcal{H}/\Lambda)$ is determined
by its action on the boundary of $\mathcal{H}/\Lambda$.
Let $\rho$ be the reflection about the axis shown in Figure \ref{fig:S12}.
Then $\rho$ reverses the orientation of the boundary of $\mathcal{H}/\Lambda$.
$\mathrm{PMod}(\mathcal{H}/\Lambda)$ is generated by the involution $\rho$.
The action of $\rho$ on $\Lambda$ is $\alpha$ and is not inner automorphism of $\Lambda$.
By Lemma \ref{L:PMod_Out0}, the order of
${\rm Out}_0(\Lambda)$ is at most $2$.
Therefore ${\rm Out}_0(\Lambda)$ is generated by $\alpha$ and
${\rm Out}_0(\Lambda)$ is isomorphic to $\mathbb{Z}_2$.
If $n \not= m$, ${\rm Out}_0(\Lambda) = {\rm Out}(\Lambda)$.
If $n = m$, there is a short exact sequence
$1 \to {\rm Out}_0(\Lambda) \to {\rm Out}(\Lambda) \to \mathbb{Z}_2 \to 1$,
where $\mathbb{Z}_2$ is the group of permutation of cones and
is generated by $\tau$ in Figure \ref{fig:S12}.
The action of $\tau$ on $\Lambda$ is $\beta$.
We conclude that ${\rm Out}(\Lambda)$ is generated by $\alpha$ and $\beta$,
and its defining relations are $\alpha^2 = \beta^2 = (\alpha \beta)^2 = 1$.
\end{pf}

\section{Proofs of the main results}\label{sec:main}
In this section we proof the results stated in Section \ref{sec:statement}, that is we classify, up to topological conjugation, cyclic actions corresponding to the signatures given in Lemma \ref{big o-r}.

\subsection{Actions with a disc with 6 corner points as the quotient
orbifold.} This is the easiest case concerning an NEC-group
$\Lambda$ with the signature
$(0;+;[\;];\{(2,2,2,2,2,2)\})$
from Lemma \ref{big o-r}. We denote
the canonical reflections $c_{1i}$ simply by $c_i$ for
$i=0,1,\dots,6$. We have $e_1=1$ and $c_0=c_6$. It follows that
$\Lambda$ has the presentation
$$
\langle c_0,\ldots ,c_5 \,|\, c_0^2 = \ldots = c_5^2 =(c_0c_{1})^2= \ldots = (c_4c_5)^2=1 \rangle.
$$

\begin{proof}[Proof of Theorem $\ref{d6}$.]
By Lemma \ref{cycle perm} there is only one, up to equivalence, BSK-epimorphism
 $\theta : \Lambda \to \zz_N$, mapping $(c_0,c_1,c_2,c_3,c_4,c_5)$ on
$(N/2,0,N/2,0,N/2,0)$. In particular, we see that $N/2$ generates
$\zz_N$, hence $N=2$. By Lemma \ref{sign k} and the Hurwitz-Riemann formula, $S$ is a 3-holed sphere.
\end{proof}

\subsection{Actions with annulus with 2 corner points as the quotient orbifold.}
This case concerns an NEC-group $\Lambda$ with the signature
$(0;+;[\;];\{(\,),(2,2)\})$
from Lemma \ref{big o-r} which has the presentation
 {$$
\langle e_1,e_2,c_{10},c_{20},c_{21},c_{22}
\,|\,
e_1e_2= c_{ij}^2=(c_{20}c_{21})^2=(c_{21}c_{22})^2=1,
e_1c_{10}=c_{10}e_1, e_2c_{20}=c_{22}e_2 \rangle.
$$ }
\begin{proof}[Proof of Theorem $\ref{ann2}$]
Let $\theta : \Lambda \to \zz_N$ be a BSK-epimorphism and let $\theta (e_1)=a$.
By Lemma \ref{cycle perm} we may assume
$$\theta(e_1)=a,\ \theta(e_2)=-a,\ \theta (c_{20})= \theta (c_{22})=N/2,\ \theta (c_{21})=0\
{\rm and} \ \theta (c_{10})=0 \; {\rm or}\; N/2.$$ Since $a$ and
$N/2$ generate $\zz_N$, we have $(a,N/2)=1$ and it follows that the
order of $a$ is either $N$ or $N/2$, the latter being possible only
for odd $N/2$.

\smallskip
Suppose that the order of $a$ is $N$. Then after composing $\theta$
with a suitable automorphism of $\zz_N$ we can assume that
$\theta(e_1)=1$. By Lemma \ref{sign k} $S$ is non-orientable, since
$e_1^{N/2}c_{20}$ is a non-orientable word in $\Gamma=\ker\theta$,
and its number of boundary components is either $N/2$ if
$\theta(c_{10})\ne 0$, or $N/2+1$ if $\theta(c_{10})=0$. From the
 Hurwitz-Riemann formula we easily compute that the genus of $S$ is
respectively $2$ or $1$.
\smallskip
Now suppose that the order of $a$ is $N/2$ which is odd. As above,
after composing $\theta$ with a suitable automorphism of $\zz_N$ we
can assume that $\theta(e_1)=2$ and therefore again we obtain two
nonequivalent BSK-maps, which give rise to two topologically
non-conjugate actions. Observe however, that this time $S$ is
orientable by Lemma \ref{sign k}, and it has either $N/2$ or $N/2+2$
boundary components. By the Hurwitz-Riemann formula the genus of $S$
is respectively $1$ or $0$.
\end{proof}

\subsection{Actions with M\"obius band with 2 corner points as the quotient orbifold.}
This case concerns an NEC-group $\Lambda$ with signature
$(1;-;[\;];\{(2,2)\})$
from
Lemma \ref{big o-r}.
We denote the canonical reflections
$c_{1i}$ simply by $c_i$ for $i=0,1,2$. After ruling out the
redundant generator $e_1$, we can write a presentation for $\Lambda$ as
$$
\langle d, c_{0},c_{1},c_{2} \,|\,
c_{0}^2=c_{1}^2=c_{2}^2=(c_{0}c_{1})^2=(c_{1}c_{2})^2=1,
c_{0}d^2=d^2c_{2}\rangle.
$$

%111111111111111111111111111
\begin{proof}[Proof of Theorem $\ref{mb2}$.]The proof is very similar to that of Theorem \ref{ann2} above.
By Lemma \ref{cycle perm} every BSK-map $\theta : \Lambda \to \zz_N$
is equivalent to one of the form
$$
\theta (c_{0})=\theta (c_{2})=N/2,\ \theta (c_{1})=0,\ \theta(d)=a,
$$
for some $a\in\zz_N$. By (b) of Lemma \ref{sign k}, $S$ has $N/2$
boundary components, and from the Hurwitz-Riemann formula we compute
that its genus if either $2$ if it is non-orientable, or $1$
otherwise. There are two cases, according to the order of $a$,
which is either $N$ or $N/2$. If $a$ has order $N$, then by
composing $\theta$ with a suitable automorphism of $\zz_N$ we can
assume that $\theta(d)=1$. Observe that here $N/2$ can be arbitrary,
but $S$ is non-orientable if and only if $N/2$ is even, since only
then $d^{N/2}c_0$ is a non-orientable word in $\ker \theta$. If the
order of $a$ is $N/2$ then it must be odd, and by composing $\theta$
with a suitable automorphism of $\zz_N$ we can assume that
$\theta(d)=2$. In such case $d^{N/2}$ is a non-orientanble word in
$\Gamma = \ker \theta$ and hence $S$ is non-orientable.
\end{proof}

\subsection{Actions with a 1-punctured disc with 2 corner points as the quotient orbifold.}
This is the case concerning an NEC-group $\Lambda$ with the
signature
$(0;+;[m];\{(2,2)\})$
from Lemma \ref{big o-r}.
We denote the
canonical reflections $c_{1i}$ simply by $c_i$ for $i=0,1,2$. After
ruling out the redundant generator $e_1$, we can rewrite the
presentation for $\Lambda$ as
$$
\langle x, c_{0},c_{1},c_{2}
\,|\,
x^m= c_{0}^2=c_{1}^2=c_{2}^2=(c_{0}c_{1})^2=(c_{1}c_{2})^2=1,
c_{0}x=xc_{2}\rangle
$$
\begin{proof}[Proof of Theorem $\ref{d12}$.] Let $\theta : \Lambda \to \zz_N$ be a BSK-epimorphism.
Then $\zz_N$ is generated by $\theta (x)$, which has order $m$, and
$\theta(c_i)$ for some canonical reflection $c_i$ which has order $2$. Thus
either $N=m$ if $m$ is even, or $N=2m$ if $m$ is odd. By Lemma
\ref{cycle perm}, we can assume that $\theta(c_0)=\theta(c_2)=N/2,\
\theta(c_1)=0$. If $m=N$ then after composing $\theta$ with a
suitable automorphism of $\zz_N$ we can assume that $\theta(x)=1$,
and hence the action is unique. By (a) of Lemma \ref{sign k}, $S$ is
non-orientable in this case, since $x^{N/2}c_0$ is a non-orientable
word in $\ker \theta$. Now assume that $N=2m$. Then after composing
$\theta$ with a suitable automorphism of $\zz_N$ we can assume that
$\theta(x)=2$, and hence also in this case the action is unique up
to topological conjugation. Observe that now $S$ is orientable, by
(a) of Lemma \ref{sign k}. Finally, by (b) of Lemma \ref{sign k},
$S$ has $N/2$ boundary components in both cases, and its genus can
be easily computed from the Hurwitz-Riemann formula.
\end{proof}

\subsection{Actions with a 1-punctured disc with 4 corner points as the quotient orbifold.}
This is the case concerning an NEC-group $\Lambda$ with the
signature
$(0;+;[m];\{(2,2,2,2)\})$
from Lemma \ref{big o-r}.
This case is very
similar to that from the previous section. Now $\Lambda$ has
 the presentation
$$
\langle x, c_{0}, \ldots , c_{4} \,|\, x^m=c_0^2= \ldots =c_4^2=
(c_{0}c_{1})^2= (c_{1}c_{2})^2=(c_{2}c_{3})^2= (c_{3}c_{4})^2=1,
c_{0}x=xc_{4}\rangle.
$$
The proof of Theorem \ref{d14} is almost identical as that of
Theorem \ref{d12}. We leave details to the reader.

\subsection{Actions with 1-punctured M\"obius band as the quotient orbifold.}
These actions correspond to an NEC-group $\Lambda$ with
signature
$(1;-;[m];\{(\,)\})$
 from  Lemma \ref{big o-r} which has the presentation
$$
\langle x, d, c, e \;|\; xed^2=x^m=c^2=1, ec=ce \rangle.
$$
We have $\mu(\Lambda)=(m-1)/m$.
\begin{proof}[Proof of Theorem $\ref{mb1ori}$.] Suppose
that $\theta\colon\Lambda\to\Z_N$ is a BSK-epimorphism, such that
$\Gamma=\ker\theta$ is an orientable bordered surface group with $k$
empty period cycles.
Since $\Gamma$ contains a reflection,  we have $\theta(c)=0$, and by (b) of Lemma \ref{sign k}, $\theta(e)$ has
order $n=N/k$. In particular, $k$ divides $N$. Note that $\theta(x)$
and $\theta(e)$ generate a subgroup of index at most $2$ of $\Z_N$.
Since $S$ is orientable, this has to be a proper subgroup
of $\mathbb{Z}_{N}$,
 for
otherwise $\Gamma$ would contain a non-orientable word of the form
$wd$, where $w$ is a word in $x$ and $e$ such that
$\theta(w)=-\theta(d)$. It follows that $N=2\mathrm{lcm}(m,n)$.

By multiplying $\theta$ by an element of $\Z_N^\ast$, we can assume
that
\begin{equation}\label{theta_a}
 \theta(x)= {N}/{m},\;\; \theta(e)= a{N}/{n}, \;\; \theta(d)=b,
\end{equation}
for some $a \in \Z_{n}^\ast$ and some $b$ for which
\begin{equation}\label{eq:2b}
{N}/{m} + a{N}/{n}+ 2b \equiv 0 \tmod{N}
\end{equation}
We  denote such BSK-map by $\theta_a$, bearing in mind that
the parameter $a$ does not always determine it uniquely. Indeed, we have
$$b=-N/2m-aN/2n+\varepsilon N/2=-n/t-am/t+\varepsilon N/2$$ for some $\varepsilon\in\Z_2$.

If $N/2$ is odd then for arbitrary $a\in\Z_n^\ast$ we have a unique
odd $b$, namely $\varepsilon=0$ if $a$ is even, and $\varepsilon=1$
if $a$ is odd. If $N/2$ is even but $t$ is odd, then $b$ is odd for
arbitrary $a\in\Z_n^\ast$ and $\varepsilon\in\Z_2$. Finally, suppose
that $t$ is even. Then $a$ must be odd since $(a,n)=1$, and $b$ is
odd if and only if $n/t$ and $m/t$ have opposite parity. But since
$n/t$ and $m/t$ are relatively prime, $b$ is odd if and only if
$mn/t^2=N/2t$ is even and $\varepsilon$ arbitrary.

Summarizing the above paragraph, we conclude that if $N/2$ is odd,
then $\theta_a(d)$ is uniquely determined by $a$, whereas if $N/2$
is even, then $\theta_a(d)$ is determined only modulo $N/2$.
However, the two different possibilities for $\theta_a(d)$ define
equivalent BSK-epimorphisms. Indeed, set $c=1+N/2$ and note that
$c\equiv 1 \tmod m$, $c\equiv 1 \tmod n$ and $c$ is odd, hence
$c\in\Z_N^\ast$. Furthermore, we have $c\theta_a(x)=N/m$,
$c\theta_a(e)=aN/n$ and $c\theta_a(d)=\theta_a(d)+N/2$, because
$\theta_a(d)$ is odd.

\smallskip
Now, we will determine the number of equivalence classes of
BSK-epimorphisms. For $\phi$ representing an element of ${\rm
Out}(\Lambda)$ and a BSK-epimorphism $\theta: \Lambda \to \Z_N$,
let $\theta^\phi= \theta\circ\phi$. Since $\theta$ is into an
abelian target, for the generators of ${\rm Out}(\Lambda)$ given in
Lemma \ref{mb1-lem} we have
$$
\theta^\gamma\colon\begin{cases}
x\mapsto -\theta(x)\\
e \mapsto -\theta(e)\\
 d\mapsto -\theta(d)\end{cases}\,
\theta^\delta\colon\begin{cases}
x\mapsto \theta(x)\\
e\mapsto -\theta(e)\\
 d\mapsto -\theta(d) - \theta(x)\end{cases}
$$
In particular, for every $\phi$ we have $\theta^\phi(x)=\pm\theta(x)$ and $\theta^\phi(e)=\pm\theta(e)$.

We claim that $\theta_a$ and $\theta_{a'}$ are equivalent if and
only if $a\equiv\pm a'\tmod t$. For suppose that
$\theta_{a'}=c\theta_a^\phi$ for some $c\in\Z_N^\ast$ and
$\phi\in{\rm Aut}(\Lambda)$. Then, by replacing $c$ by $-c$ if
necessary, we may assume that $\theta_{a'}(x)=c\theta_a(x)$ and
$\theta_{a'}(e)=\pm c\theta_a(e)$, which gives $c\equiv 1\tmod m$
and $a'\equiv\pm ca\tmod n$. It follows that $a\equiv \pm a'\tmod
t$. Conversely, suppose that $a\equiv \pm a'\tmod t$. Then by Lemma
\ref{CRT} there exists $c$, such that $c\equiv 1\tmod m$ and
$c\equiv\pm a'a^{-1}\tmod n$, where $a^{-1}$ denotes the inverse of
$a$ modulo $n$. Note that such $c$ is relatively prime to $m$ and
$n$, and hence to $N/2$. If $N/2$ is odd, then we can take $c$ to be
odd as well, so that $c\in\Z_N^\ast$. Now
$\theta_{a'}(x)=c\theta_a^\phi(x)$ and
$\theta_{a'}(e)=c\theta_a^\phi(e)$, where $\phi=\mathrm{id}$ or
$\phi=\delta$. If $N/2$ is odd, then necessarily also
$\theta_{a'}(d)=c\theta_a^\phi(d)$, whereas if $N/2$ is even, than
possibly $c\theta_a^\phi(d)=\theta_{a'}(d)+N/2$, in which case it
suffices to replace $c$ by $c+N/2$.

Summarizing, on one hand each element of $\Z_t^\ast$ is the
residue mod $t$ of some $a\in \Z_{n}^\ast$ defining BSK-map
$\theta_a$ by Lemma \ref{DT}. On the other hand $\theta_a$ and
$\theta_{a'}$ are equivalent if and only if $a \equiv \pm a'
\tmod{t}$. So the elements of the quotient group $\Z_t^\ast/\{\pm
1\}$ parametrise the equivalence classes of BSK-maps (although it might happen that for a particular representative $x \in \Z_t^\ast$, $\theta_{x}$ is not a BSK-epimorphism, because $x\notin\zz_n^\ast$). Therefore, {we have $ \varphi(t)/2
$ classes for $t>1$ and $1$ class for $t=1$}.
\end{proof}

\begin{proof}[Proof of Theorem $\ref{mb1nonori}$.] Suppose that $\theta\colon\Lambda\to\Z_N$ is a BSK-epimorphism, such that
$\Gamma=\ker\theta$ is a non-orientable bordered surface group with
$k$ empty period cycles. As in the proof of the previous
theorem, we have $\theta(c)=0$ and $\theta(e)$ has order
$n=N/k$. Since $\Gamma$ contains a non-orientable word, $\theta(d)$
is equal to $\theta(w)$ for some word $w$ in $x$ and {$e$}. It follows
that $\Z_N$ is generated by $\theta(x)$ and $\theta(e)$, hence
$N=\mathrm{lcm}(m,n)$. By multiplying $\theta$ by an element of
$\Z_N^\ast$, we can assume that $\theta= \theta_a$ defined by
(\ref{theta_a})
 for some $a \in \Z_{n}^\ast$ and some $b$ for which
(\ref{eq:2b}) is satisfied.

\s Now, if $N$ is odd then for arbitrary $a\in\Z_n^\ast$ we have a
unique $b$ satisfying \eqref{eq:2b}.
 By the
same argument as in the previous proof, $\theta_a$ is equivalent to
$\theta_{a'}$ if and only if $a\equiv \pm a'\tmod t$, and hence
there are $\varphi(t)/2$ equivalence classes of BSK-epimorphisms if
$t>1$, and one such class if $t=1$.

\smallskip
For the rest of the proof assume that $N$ is even.
By \eqref{eq:2b},
$N/m+ a N/n=n/t+am/t$ is even, which is possible if and only if
$n/t$ and $m/t$ are both odd, hence $nm/t^2=N/t$ is odd. Now
$\theta_a(d)$ is determined by $a$ only modulo ${N}/{2}$:
$$
b=-\frac{1}{2}\left(\frac{N}{m}+a\frac{N}{n}\right) + \varepsilon\frac{N}{2}
$$
for some $\varepsilon \in \Z_2$.
We claim
that given $a,a' \in \Z_n^\ast$, $\theta_a$ and $\theta_{a'}$ are
equivalent if and only if either

\begin{itemize}
\item[(1)] $a\equiv a'\tmod t$ and $\theta_{a'}(d)=c\theta_a(d)$,
where $c$ is the unique element of $\Z_N^\ast$ satisfying $c\equiv
1\tmod {m}$ and $ca\equiv a'\tmod {n}$, or
\item[(2)] $a\equiv -a'\tmod t$ and $\theta_{a'}(d)=c\big(\theta_a(d)+\theta_a(x)\big)$, where $c$ is the unique
element of $\Z_N^\ast$ satisfying $c\equiv -1\tmod {m}$ and
$ca\equiv a'\tmod {n}$.
\end{itemize}

To prove the claim suppose that $\theta_{a'}=c\theta_a^\phi$ for
some $c\in\Z_N^\ast$ and $\phi\in\mathrm{Aut}(\Lambda)$. By Lemma
\ref{mb1-lem},  we may suppose that
$\phi\in\{1,\gamma,\delta,\delta\gamma\}$. If $\phi=1$ or
$\phi=\gamma$, then after replacing $c$ by $-c$ in the latter case,
we have $\theta_{a'}(x)=c\theta_a(x)$, $\theta_{a'}(e)=c\theta_a(e)$
and $\theta_{a'}(d)=c\theta_a(d)$. Thus $c$ satisfies $c\equiv
1\tmod {m}$ and $ca\equiv a'\tmod {n}$. By Lemma \ref{CRT}, such
(unique) $c$ exists if and only if $a \equiv a' \tmod t$. Similarly,
if $\phi=\delta$ or $\phi=\delta\gamma$, then after replacing $c$ by
$-c$ in the former case, we have $\theta_{a'}(x)=-c\theta_a(x)$,
$\theta_{a'}(e)=c\theta_a(e)$ and
$\theta_{a'}(d)=c\big(\theta_a(d)+\theta_a(x)\big)$. Such (unique)
$c$ again exists if and only if $a\equiv -a'\tmod t$. This completes
the proof of the claim.

\smallskip
Suppose $t>2$. It follows from the previous paragraph that
there is a surjection $\pi$ from the set of equivalence classes of
BSK-maps onto $\Z_t^\ast/\{\pm 1\}$, defined by
 $\pi([\theta_a])=[[a]_t]$, where $a \in \Z_n^\ast$.
We claim that $\pi$ is a 2-over-one map.
For let $\theta_a$ be a
BSK-map defined by (\ref{theta_a}) and define $\theta'_a$ by
$$
\theta'_a(x)={N}/{m},\;\; \theta'_a(e)=a{N}/{n}, \;\;
\theta'_a(d)=b+{N}/{2}.
$$
Evidently $\pi(\theta_a)=\pi(\theta'_a)$, but $\theta_a$ is not equivalent to $\theta'_a$. For if they were equivalent, then (1) would be satisfied with $c=1$, hence $b=b+N/2$. Now if $\pi(\theta_{a'})=\pi(\theta_a)$ for some $a'\in\zz_n^\ast$, then $\theta_{a'}$ is equivalent either to $\theta_a$ or to $\theta'_a$, by (1) if $a'\equiv a\tmod{t}$, or by (2) if $a'\equiv-a\tmod{t}$.

\smallskip
Finally, suppose $t=2$. By (1) every BSK map is equivalent to
$\theta\colon\Lambda\to\Z_N$ such that $\theta(x)={N}/{m}$ and
$\theta(e)={N}/{n}$. Fix such $\theta$ and define $\theta'$ by
$\theta'(x)=\theta(x)$, $\theta'(e)=\theta(e)$, and
$\theta'(d)=\theta(d)+{N}/{2}$. We have to show that $\theta$ and
$\theta'$ are equivalent. Let $c$ be the unique element of
$\Z_N^\ast$ such that $c\equiv -1\tmod {m}$ and $c\equiv 1\tmod
{n}$. By (2) it suffices to show that
$\theta'(d)=c\big(\theta(d)+\theta(x)\big)$. We have
$$2c\theta(d)=-c\big(\theta(x)+\theta(e)\big)=\theta(x)-\theta(e)=2\big(\theta(x)+\theta(d)\big)$$
Either $c\theta(d)=\theta(d)+\theta(x)$ or
$c\theta(d)=\theta(d)+\theta(x)+{N}/{2}$. The former equality is not
possible, because $\theta(x)={N}/{m}$ is odd and $\theta(d)(c-1)$ is
even. Hence
$$c\big(\theta(d)+\theta(x)\big)=c\theta(d)-\theta(x)=\theta(d)+{N}/{2}=\theta'(d)$$
It follows that all BSK-maps $\Lambda\to\Z_N$ are equivalent.
\end{proof}

\subsection{Actions with a $2$-punctured disc as the quotient orbifold.}
This case concerns an NEC-group $\Lambda$ with  signature
$(0;+;[m,n];\{(\,)\})$
 from Lemma \ref{big o-r} which has the presentation
 $$
 \langle
 x_1,x_2, c,e \; | \; x_1^m=x_2^n=c^2=x_1x_2e=1, ec=ce \rangle.
 $$
We have $\mu(\Lambda)=1-1/m-1/n$.

\begin{proof}[Proof of Theorem $\ref{d21}$.] Suppose that $\theta\colon\Lambda\to\zz_N$ is a BSK-map. Since
$\ker\theta$ contains a reflection, $\theta(c)=0$ and it follows by
(a) of Lemma \ref{sign k} that $S$ is orientable. By (b) of Lemma
\ref{sign k} we have $k=N/l$, where $l$ is the order of $\theta(e)$.
Since $\theta$ is a surjection, $\zz_N$ is generated by
$\theta(x_1)$, $\theta(x_2)$ which have orders $m$ and $n$
respectively and $\theta(e)= -\big(\theta(x_1)+\theta(x_2)\big)$ has
order $l$. It follows that the conditions (i) and (ii) of Lemma
\ref{lem-Harv} are satisfied, in particular $N=\mathrm{lcm}(m,n)$.
Let $(A,A_1,A_2,A_3)$ be the Maclachlan decomposition of $(m,n,l)$,
as above. We have $t=AA_3$, $A_1=n/t$, $A_2=m/t$ and $A_3=k$. We see
that $k$ divides $t$, and because $A_3$ is relatively prime to
$A_1A_2$, $k$ is relatively prime to $nm/t^2=N/t$. It follows that
$k$ divides $ {t}/{(t,N/t)}$. Finally, if $N$ is even, then by (ii)
of Lemma \ref{lem-Harv}, one of the numbers $k$, $n/t$, $m/t$ must
be even. It follows that $k$ must be even if $N/t=nm/t^2$ is odd.
Conversely, having $k,n,m,N$ satisfying the conditions of the theorem, one can easily define, using Lemma \ref{lem-Harv}, an appropriate BSK-map defining a surface and an action in question.

\smallskip
Every BSK-map is equivalent to $\theta_a\colon\Lambda\to\zz_N$ defined by
$$
\theta_a(x_1)=A_1,\quad
\theta_a(x_2)=aA_2,\quad \theta_a(e)=-(A_1+aA_2),\quad\theta_a(c)=0
$$
for some $a\in\mathcal{L}$, where
$$
\mathcal{L}=\{a\in\zz_n^\ast\;|\;\textrm{$A_1+aA_2$\ has\ order\ $l$}\}.
$$
Suppose first that $m\ne
n$. Let $S=\{c\in\zz_N^\ast\;|\;c\equiv 1\tmod{m}\}$. For
$a,b\in\mathcal{L}$, we claim that $\theta_a$ and $\theta_b$ are
equivalent if and only if $b\equiv ca\tmod{n}$ for some $c\in S$.
Indeed, suppose that $\theta_b=c\theta_a^\phi$ for some
$\phi\in\mathrm{Aut}(\Lambda)$ and $c\in\zz_N^\ast$. By Lemma
\ref{d2-lem}, either $\theta_a(\phi(x_i))=\theta_a(x_i)$ for
$i=1,2$, or $\theta_a(\phi(x_i))=-\theta_a(x_i)$ for $i=1,2$. By
changing $c$ to $-c$ in the latter case, we have $A_1=cA_1$ and
$bA_2=caA_2$, and the claim follows. Thus, the equivalence classes
of BSK-maps are parametrised by the orbits of the action of $S$ on
$\mathcal{L}$. Since this action is free, the number of orbits is
$|\mathcal{L}|/|S|$. Let $B=A/C=t/{kC}$ and write $B=B_1B_2B_3$,
where for $i=1,2,3$ each prime dividing $B_i$ divides $A_i$. By
\cite[Theorem 3.4]{BCGH} we have
$$|\mathcal{L}|=\varphi(A_1B)\psi(C)=\varphi(A_1B_1)\varphi(B_2)\varphi(B_3)\psi(C).$$
We also have
\begin{align*}
&|S|=\frac{\varphi(N)}{\varphi(m)}=\frac{\varphi(A_1B_1)\varphi(A_2B_2)\varphi(A_3B_3)\varphi(C)}
{\varphi(B_1)\varphi(A_2B_2)\varphi(A_3B_3)\varphi(C)}=\frac{\varphi(A_1B_1)}{\varphi(B_1)}\\
&|\mathcal{L}|/|S|=\varphi(B_1)\varphi(B_2)\varphi(B_3)\psi(C)=\varphi(B)\psi(C).
\end{align*}
This completes the proof in the case $m\ne n$.

\smallskip
Now suppose that $m=n$. {This} common value is equal to $N$ and
we have $A_1=A_2=1$ and
$$
\mathcal{L}=\{a\in\zz_N^\ast\;|\;\textrm{$1+a$\ has\ order\ $l$}\}.
$$
Now $\theta_a(\beta(x_1))=\theta_a(x_2)$ and
$\theta_a(\beta(x_2))=\theta_a(x_1)$ for $\beta
\in\mathrm{Aut}(\Lambda)$ from Lemma \ref{d2-lem}. Consequently,
$\theta_a$ and $\theta_b$ are equivalent if and only if either $a=b$
or $ab=1$. It follows that the number of equivalence classes of
BSK-maps is $(|\mathcal{L}|+I)/2$, where $I$ is the number of
$a\in\mathcal{L}$ for which $a^2=1$
 As in the case $m\ne n$, we have
$|\mathcal{L}|=\varphi(B)\psi(C)$, where $C$ is the biggest divisor
of $l$ coprime with $k$, and $B=l/C$.

\smallskip
In order to compute $I$, suppose that $a^2=1$ for some
$a\in\mathcal{L}$. We have $N=kBC$, and since $kB$ and $C$ are
coprime, $\zz_N \cong \zz_{kB}\oplus\zz_C$. Under this isomorphism, we
write $a=(a_1,a_2)$, where $a_1\in\zz_{kB}^\ast$ and
$a_2\in\zz_C^\ast$. We have $a_1^2\equiv 1\tmod{kB}$ and
$a_2^2\equiv 1\tmod{C}$. Since $1+a$ has order $l$, we have
$1+a_1=ks$ for some $s\in\zz_B^\ast$ and $1+a_2\in\zz_C^\ast$.

\smallskip
In suitable rings we have
$$0=1-a_i^2=(1+a_i)(1-a_i).$$
Since $(1+a_2)$ is invertible in $\zz_C$, $a_2=1$.
In $\zz_{kB}$ we have
$$0=(1+a_1)(1-a_1)=ks(2-ks).$$
Since $s$ is invertible, it follows that $B$ divides $2-ks$, hence
$(B,k)\le 2$. Observe that every prime divisor of $B$ divides $k$,
hence also $(B,k)$. It follows that $B$ is a power
of $2$. If $B\le 2$ then $s=1$ and $a=(a_1,a_2)=(k-1,1)$.
If $B=2^z$ for
$z>1$, then $k/2$ is coprime to $B/2$. Let $k'\in\zz_{B/2}^\ast$
denote the inverse of $k/2$. Then, since $B/2$ divides $1-sk/2$, we
have $s\equiv k'\tmod{B/2}$, and hence $s=k'$ or $s=k'+B/2$.
Summarising, we have
$$
I=\begin{cases}
2&\textrm{for\ } B=2^z,\ z>1\\
1&\textrm{for\ } B\le 2,\\
0&\textrm{otherwise.}
\end{cases}
$$
To finish the proof, observe that since $\psi(C)$ is odd, and
$\varphi(B)$ is even if and only if $B>2$, we have
$I\equiv\varphi(B)\psi(C)\pmod{2}$. It follows that
$(|\mathcal{L}|+I)/2=\lceil|\mathcal{L}|/2\rceil$ if $I\ne 2$.
\end{proof}

\subsection{Actions with a 1-punctured annulus as the quotient orbifold.}
This case concerns an NEC-group $\Lambda$ with the signature
$(0;+;[m];\{(\,),(\,)\})$
from
 Lemma \ref{big o-r} which has the presentation
$$\langle x,e_1,e_2,c_1,c_2 \; | \;
xe_1e_2=x^m=c_1^2=c_2^2=1, e_1c_1=c_1e_1, e_2c_2=c_2e_2
\rangle.
$$
We have $\mu(\Lambda)=(m-1)/m$.

\begin{proof}[Proof of Theorem $\ref{ann1-nonori}$.] Suppose that $\theta\colon\Lambda\to\zz_N$ is a BSK-map, such that
$\Gamma= \ker\theta$ is a non-orientable bordered surface group.
Then, since $\Gamma$ is bordered, some of the canonical
reflections, say $c_1$ belongs to $\Gamma$. Furthermore, by (b) of
Lemma \ref{sign k}, the order of $\theta(e_1)$ is $N/k$. But since
$\Gamma$ is non-orientable, it contains a non-orientable word, by
(a) of Lemma \ref{sign k}, which is possible if and only if
$\theta(c_2)=N/2$ and $N/2$ is in the subgroup of $\zz_N$ generated
by $\theta(x)$ and $\theta(e_1)$, and thus we obtain the
condition $N=\mathrm{lcm}(m,N/k)$.
Conversely, if the last condition is satisfied, then for
$a\in\zz_{N/k}^\ast$ we can define a BSK-map
$$\theta_a(x)={N}/{m},\;\theta_a(e_1)=ak,\;\theta_a(e_2)=-\big({N}/{m}+ak\big),\;\theta_a(c_1)=0,\;\theta_a(c_2)={N}/{2},$$
and every BSK-map is equivalent to some $\theta_a$. Let
$a,a'\in\zz_{N/k}^\ast$ and suppose that
$\theta_{a'}=c\theta_a\phi$ for some $c\in\zz_N^\ast$
and some $\phi\in\mathrm{Aut}(\Lambda)$. By Lemma \ref{ann1-lem},
$\theta_a\phi$ maps $(x,e_1)$ on $\pm (N/m,ak)$ and so by replacing
$c$ by $-c$ if necessary in the latter case, we obtain that
$$\theta_{a'}(x)=c{N}/{m},\;\theta_{a'}(e_1)=cak,\;\theta_a(e_2)=-c\big({N}/{m}+ak\big),\;\theta_a(c_1)=0,\;\theta_a(c_2)={N}/{2},$$
which give $c\equiv 1\tmod m$ and $a'\equiv ca\tmod{N/k}$. As in
the proof of Theorem \ref{mb1ori}, we conclude that
$\theta_a$ and $\theta_{a'}$ are equivalent if and only if $a'\equiv
a\tmod t$, and hence, the number of equivalence classes of such
BSK-maps is $\varphi(t)$.
\end{proof}

As we already mentioned in Section 2, the case of orientable $S$ is much more involved.

\begin{proof}[Proof of Theorem $\ref{ann1}$.]
Suppose that an action exists and let
$\theta\colon\Lambda\to\zz_N$ be the corresponding BSK-map. Since
$S$ is bordered, $\ker\theta$ contains at least one canonical
reflection, and we can assume $\theta(c_1)=0$. We consider two
cases: (1) $\theta(c_2)\ne 0$ and (2) $\theta(c_2)=0$.

\smallskip \noindent {\bf Case 1:} $\theta(c_2)\ne 0$. By (b) of
Lemma \ref{sign k}, $\theta(e_1)$ has order $N/k$, in particular
$k|N$. Since $S$ is orientable, $\theta(x)$ and $\theta(e_1)$
generate the subgroup of index 2 of $\zz_N$, by (a) of Lemma
\ref{sign k}. Hence $N=2\textrm{lcm}(m,N/k)$, and $\theta(c_2)=N/2$
is odd. Conversely, if (1) is satisfied, then for each
$a\in\zz_{N/k}^\ast$ we can define a BSK-map
$\theta^{1}_a\colon\Lambda\to\zz_N$ by
$$\theta^{1}_a(x)=N/m,\;\theta^{1}_a(e_1)=ak,\;\theta^{1}_a(e_2)=-(N/m+ak),
\;\theta^{1}_a(c_1)=0,\;\theta^{1}_a(c_2)=N/2.$$

\smallskip
\noindent {\bf Case 2:}
$\theta(c_2)= 0$.
Let $l_i$ denote the
order of $\theta(e_i)$ and set
$n_i=N/l_i$ for $i=1,2$. Then, by (b)
of Lemma \ref{sign k} we have $n_1+n_2=k$. Now $\theta(x)$,
$\theta(e_1)$ and $\theta(e_2)$ generate $\zz_N$ and hence the
triple $(l_1,l_2,m)$ satisfies the conditions of Lemma
\ref{lem-Harv}. Consider the Maclachlan decomposition of
$(l_1,l_2,m)$
$$
l_1=AA_2A_3,\ l_2=AA_1A_3,\ m=AA_1A_2.
$$
We have $A_1=n_1=n$, $A_2=n_2$, $A_3=N/m$ and the conditions
(a),(b),(c) follow from the properties of the Maclachlan
decomposition. Conversely, if (2) is satisfied, then by Lemma
\ref{lem-Harv}, $\zz_N$ is generated by three elements $a$, $b$ and
$c$ of orders $m$, $l_1$ and $l_2$ respectively such that $a+b+c=0$.
We define $\theta^{2}\colon\Lambda\to\zz_N$ by
$$
\theta^{2}(x)=a,\;\theta^{2}(e_1)=b,\;\theta^{2}(e_2)=c,\;\theta^{2}(c_1)=\theta^{2}(c_2)=0.$$
This completes the proof of assertion (i).

Now we shall find the number of conjugacy classes of actions. The
proof of assertion (ii) is analogous to that of Theorem
\ref{ann1-nonori} and we omit it.

To prove (iii), consider the Maclachlan decomposition
$(A,A_1,A_2,A_3)$ of the triple $(N/n,N/(k-n),m)$. We have $A_1=n$,
$A_2=k-n$, $A_3=N/m$ and $A=m/(n(k-n))$.
Every BSK-map is equivalent to $\theta^{2}_a\colon\Lambda\to\zz_N$
 defined by
$$\theta^{2}_a(e_1)=A_1,\;\theta_a^{2}(e_2)=aA_2,\;\theta^{2}_a(x)=-(A_1+aA_2),
\;\theta^{2}_a(c_1)=\theta^{2}_a(c_2)=0.$$
for some $a\in\mathcal{L}$, where
$$\mathcal{L}=\{a\in\zz_{N/(k-n)}^\ast\;|\;\textrm{$A_1+aA_2$\ has\ order\ $m$}\}.$$
It follows form Lemma
\ref{ann1-lem}, that for $a,a'\in\mathcal{L}$, $\theta_a^{2}$ is
equivalent to $\theta_{a'}^{2}$ if and only if $a'=ca$ for some
$c\equiv 1\tmod{N/n}$, or $aa'=1$, the latter being possible only
for $n=n-k=1$. Now, the formulas for the number of equivalence
classes of BSK-maps can be obtained by repeating the calculations
from the proof of Theorem \ref{d21}.
The assertion (iv) is evident.
\end{proof}

Theorem \ref{ann1} has some delicate subtlety which we
illustrate with two remarks and two examples.

\begin{rk}
\rm For some triples $(N,m,k)$ both conditions (1) and (2) are
satisfied, as for instance in Example \ref{ann1-ex1} below. In such
a case, $\theta^{1}$ and $\theta^{2}$ are not equivalent. For
suppose that $\theta^{1}=c\theta^{2}\phi$ for some
$c\in\zz_N^\ast$ and some $\phi\in\mathrm{Aut}(\Lambda)$. By Lemma
\ref{ann1-lem}, $\phi$ preserves $\{c_1,c_2\}$, hence
$\theta^1(c_i)=c\theta^{2}\phi(c_i)=0$ for $i=1,2$. This is a contradiction, because $\theta^1(c_1)\ne\theta^1(c_2)$ .
\end{rk}

\begin{rk}
\rm Suppose that $N, m, k$ satisfy the condition (2) , and let
$\theta_1^2$ and $\theta_2^2$ be BSK-maps, where $\theta_1^2(e_1)$,
$\theta_1^2(e_2)$ have orders $N/n_1$, $N/n_2$, where $n_1+n_2=k$
and $\theta_2^2(e_1)$, and $\theta_2^2(e_2)$ have orders $N/n'_1$,
$N/n'_2$, where $n'_1+n'_2=k$. If $\{n_1, n_2\}\ne\{n'_1, n'_2\}$
then $\theta_1^2$ and $\theta_2^2$ are not equivalent. This follows
from Lemma \ref{ann1-lem}, because for every
$\phi\in\mathrm{Aut}(\Lambda)$, $\theta_1^2\phi$ maps $(e_1,e_2)$
 on $\pm \big(\theta_1^2(e_1), \theta_1^2(e_2)\big).$
\end{rk}

\begin{ex}\label{ann1-ex1}
Suppose $k=2$, $2m\mid N$ and $N/2$ is odd. Then both conditions (1)
and (2) are satisfied ($n=1$ in (2)). The number of BSK-maps of
type (1) is $\varphi(m)$ by the assertion (ii), and the number of BSK-maps of type (2) is
$\left\lceil\varphi(m/C)\psi(C)/2\right\rceil$, where $C$ is the
biggest divisor of $m$ coprime to $N/m$, by the assertion (iii). By
adding up these two numbers we obtain the total number of
topological types of $\zz_N$-action on $S$, with the prescribed
quotient orbifold. By the Hurwitz-Riemann formula, the genus of $S$
is $N(m-1)/{2m}$.
\end{ex}

\begin{ex}\label{ann1-ex2}
Consider $m=N=12$ and $k=7$. Then (1) is not satisfied, but (2) is
by two different pairs $\{n,k-n\}$, namely $\{1,6\}$ and $\{3,4\}$.
By the assertion (iii) of Proposition \ref{ann1}, for each of these pairs, the
corresponding BSK-map is unique up to equivalence. Thus we have two
different topological types of $\zz_{12}$-action on $S$, with the
prescribed quotient orbifold. By the Hurwitz-Riemann formula, the
genus of $S$ is $3$.
\end{ex}

\subsection{Actions with a 3-punctured disc as the quotient orbifold.}
This subsection concerns NEC-groups $\Lambda$ with the signatures
$$(0;+;[2,3,m];\{(\,)\}) \; {\rm for}\;   m=3,4,5 \;\; {\rm and} \;\; (0;+;[2,2,m];\{(\,)\})$$
from Lemma \ref{big o-r} which have the presentation
$$
\langle x_1,x_2,x_3, e, c \,|\, x_1^2=x_2^n=x_3^m= x_1x_2x_3 e=
c^2=1, ce=ec \rangle,
$$
where $n=2$ or $n=3$.

\begin{proof}[Proof of Theorem $\ref{d31}$.] Suppose
that $\theta\colon\Lambda\to\zz_N$ is a BSK-map, such that
$\ker\theta$ is a bordered surface group. Then $\theta(c)=0$,
$\theta(x_1)=\theta(x_2)=N/2$, and by multiplying $\theta$ by an
element of $\zz_N^\ast$, we may assume $\theta(x_3)=N/m$, and hence
$\theta(e)=-N/m$. Evidently, such BSK-map is unique up to
equivalence. Since $\theta$ is an epimorphism, we have
$N=\mathrm{lcm}(2,m)$. By Lemma \ref{sign k}, $S$ is
orientable and has $N/m$ boundary components. The genus of $S$ is uniquely determined by the Hurwitz-Riemann formula.
\end{proof}

\begin{proof}[Proof of Theorem $\ref{d32}$.] Here
the cyclic group $\zz_N$ is generated by three elements of orders
$2,3$ and $m$ and hence $N=\mathrm{lcm}(2,3,m)$. For any BSK-map
$\theta\colon\Lambda\to\zz_N$ we have $\theta(c)=0$, and it follows
from Lemma \ref{sign k}, that $S$ is orientable and its number of
boundary components is $N/l$, where $l$ is the order of
$\theta(e)=-(\theta(x_1)+\theta(x_2)+\theta(x_3))$. We have
$\theta(x_1)=N/2$, and by multiplying $\theta$ by a suitable element
of $\zz_N^\ast$ we may assume $\theta(x_2)=N/3$.

\smallskip
For $\Lambda$ with the signature $(0;+;[2,3,3];\{(\,)\})$, any
BSK-epimorphism $\theta : \Lambda \to \zz_{6}$ is equivalent to one
mapping $(x_1,x_2,x_3)$ either on $(3,2,2)$ or $(3,2,4)$. In the
former case we have $\theta(e)=5$ and $S$ has $1$ boundary component
and genus $3$. In the later case we have $\theta(e)=3$ and $S$ has
$3$ boundary component and genus $2$.

\smallskip
If $\Lambda$ has signature $(0;+;[2,3,4];\{(\,)\})$ or
$(0;+;[2,3,5];\{(\,)\})$, then by Chinese reminder theorem, there is
$c\in\zz_N^\ast$ such that $c\theta$ maps $(x_1,x_2,x_3) $ on
$(N/2,N/3,N/m)$. In both cases we have $\theta(e)=-1$, hence $S$ has
one boundary component. The genus of $S$ is easily computed from the Hurwitz-Riemann formula.
\end{proof}

\subsection{Actions with a 2-punctured disc with two corners as the quotient orbifold.}
This case concerns NEC-groups $\Lambda$ with  signatures
$$(0;+;[3,m];\{(2,2)\})\; {\rm for}\;  m=3,4,5 \; {\rm and} \;  (0;+;[2,m];\{(2,2)\})$$
from Lemma \ref{big o-r} which have the presentation
$$
\langle x_1,x_2, e, c_0, c_1, c_2 \,|\,
x_1^n=x_2^m=c_0^2=c_1^2=c_2^2=(c_0c_1)^2=(c_1c_2)^2=x_1x_2e=1,
c_2e=ec_0 \rangle,
$$
where $n=2$ or $n=3$.

\begin{proof}[Proof of Theorem $\ref{d221}$.] Suppose that $\theta\colon\Lambda\to\zz_N$ is a BSK-map, such that
$\ker\theta$ is a bordered surface group. By Lemma \ref{cycle perm},
we may assume that $\theta(c_0)=\theta(c_2)=N/2$ and
$\theta(c_1)=0$. We have $\theta(x_1)=N/2$, and by multiplying
$\theta$ by an element of $\zz_N^\ast$, we may assume
{$\theta(x_2)=N/m$},
and hence $\theta(e)=N/2-N/m$. Evidently, such
BSK-map is unique up to equivalence. Since $\theta$ is an
epimorphism, we have $N=\mathrm{lcm}(2,m)$. By Lemma
\ref{sign k}, $S$ has $N/2$ boundary components and is
non-orientable, as $x_1c_0$ is a non-orientable word in
$\ker\theta$. The genus of $S$ is uniquely determined by the
Hurwitz-Riemann formula.
\end{proof}

\begin{proof}[Proof of Theorem $\ref{d222}$.] Suppose that $\theta\colon\Lambda\to\zz_N$ is a BSK-map, such that
$\ker\theta$ is a bordered surface group. As in the previous proof,
we can assume $\theta(c_0)=\theta(c_2)=N/2$, $\theta(c_1)=0$ and
$\theta(x_1)=N/3$. Since $\theta$ is an epimorphism, $\zz_N$ is
generated by three elements of orders $2,3$ and $m$ and hence
$N=\mathrm{lcm}(2,3,m)$. By (b) of Lemma \ref{sign k}, $S$ has $N/2$
boundary components.

\smallskip
For $\Lambda$ with the signature {$(0;+;[3,3];\{(\,)\})$}, any
BSK-epimorphism $\theta : \Lambda \to \zz_{6}$ is equivalent to one
mapping $(x_1,x_2)$ either on $(2,2)$ or $(2,4)$. By \cite[Lemma 4.6
and Proposition 4.14]{BCCS}, for every
$\phi\in\mathrm{Aut}(\Lambda)$, $\phi(e)$ is conjugate to $e$ or
$e^{-1}$, and hence the order of $\theta(e)$ is an equivalence
invariant. It follows that the two maps described above are not
equivalent. Indeed, for the first one $\theta(e)=2$ has order $3$,
whereas for the second one $\theta(e)=0$ has order $1$. In both
cases $S$ is orientable, because $\theta(x_1)$, $\theta(x_2)$ are
even, whereas $\theta(c_0)=3$ is odd, and hence there is no
non-orientable word in $\ker\theta$.

\smallskip
If $\Lambda$ has signature {$(0;+;[3,4];\{(\,)\})$} or
{$(0;+;[3,5];\{(\,)\})$}, then by Chinese reminder theorem, there is
$c\in\zz_N^\ast$ such that $c\theta$ maps $(x_1,x_2) $ on
$(N/3,N/m)$. It follows that $\theta$ is unique up to equivalence.
For $m=4$ we have a non-orientable word $x_2^2c_0$ in $\ker\theta$,
hence $S$ is non-orientable. For $m=5$ there is no such word, hence
$S$ is orientable. The genus of $S$ is easily computed from the
Hurwitz-Riemann formula.
\end{proof}

\section{On uniqueness of actions realizing the solutions of the minimum genus and maximum order problems}\label{sec:minmax}

Throughout the rest of the paper the letter $p$, used before to
denote algebraic genus, will be used also to denote a prime integer,
which will not lead to any ambiguity; for the genus we will assume $p\geq 2$.

\subsection{The minimum genus and maximum order problems for finite groups acting on bordered surfaces.}
We start with the following easy proposition which justifies later
definitions.
\begin{prop}\label{existence}
Let $G$ be a finite group. Then there exists a bordered topological
surface $S$, which can be assumed to be orientable or not, such that $G$ acts on $S$ by homeomorphisms.
Furthermore, if $S$ is assumed to be orientable, then the action of
$G$ can be chosen
to contain orientation reversing
elements if and only if $G$ has a subgroup $G'$ of index $2$.
\end{prop}
\begin{pf}Let $g_1, \ldots, g_r$ be a set of generators
of $G$ and let $g_{r+1} = (g_1 \ldots g_r)^{-1}$. Clearly we can
assume that $r \geq 2$.
Let $\Lambda$ be an NEC-group with the
signature $(0;+;[\;]; \{(\,), \stackrel{r+1}{\ldots}\,,(\,)\})$ and let
us define an epimorphism $\theta: \Lambda \to G$ mapping $e_i$ to $g_i$
for $i \leq r+1$ and all $c_i$ to $1$, the identity element of $G$. Then, for $\Gamma =
\ker \theta$, we have $G\cong \Lambda/\Gamma$ acting as a group of
conformal automorphisms on $S={\mathcal H}/\Gamma$, which has the
conformal structure of orientable bordered Klein surface inherited
from the hyperbolic plane ${\mathcal H}$.
Now assume that $G$ contains a subgroup $G'$ of index $2$. Assume
that $G'$ is generated as above, let $x \in G\setminus G'$, consider
an NEC-group $\Lambda $ with signature $(1;-;[\;]; \{(\,),
\stackrel{r+2}{\ldots}\,,(\,)\})$ and define an epimorphism $\theta:
\Lambda \to G$ mapping all reflections $c_i$ to $1$, $e_i$
to $g_i$ for $1 \leq i \leq r+1$, $e_{r+2}$ to {$x^{-2}$}, and
$d_1$ to {$x$}. Then for $\Gamma = \ker \theta$, we have $G\cong
\Lambda/\Gamma$ acting as a group of conformal or anticonformal
automorphisms on $S={\mathcal H}/\Gamma$,
where $x$ reverses orientation.
If we need the action on {a} non-orientable surface, then it is
sufficient to take an NEC-group $\Lambda$ with signature $(1;-;[\;];
\{(\,), \stackrel{r+1}{\ldots}\,,(\,)\})$ and define $\theta$ on
$e_i, c_i$ as above and $\theta(d_1)=1$, in virtue of
(a) of Lemma \ref{sign k}.
\end{pf}

So, let ${\mathcal K}_{+}(N)$ (resp. ${\mathcal K}_{-}(N)$) be the
family of orientable (resp. non-orientable) bordered topological
surfaces,
admitting a self-homeomorphism of order $N$. Denote by $p=p(S)$
the algebraic genus of a bordered surface $S$ and recall that it is
{the rank of the fundamental group of $S$ }
equal to $\varepsilon g +k-1$, where $g$ is the topological genus of $S$,
$k$ is the number of its boundary components and $\varepsilon=2$ or $1$
if $S$ is orientable or not. By $S_p^\pm$ will be denoted a
bordered surface of algebraic genus $p$, orientable if the sign is
$+$ and non-orientable otherwise, and similar meaning will have
$S^\pm_{g,k}$

\s Denote by ${\rm H}_{\rm p}(S)$ the set of all periodic
self-homeomorphisms of $S$ and consider two of its subsets ${\rm
H}_{\rm p}^{+}(S)$ and ${\rm H}_{\rm p}^{-}(S)$, consisting of respectively
orientation-preserving and orientation-reversing
self-homeomorphisms when $S$ is orientable. Finally let
$$
\begin{tabular}{l}
${\mathcal K}_{+}^{+}(N) = \{ S\in {\mathcal K}_{+}(N) \, |\,
 { \text{there is an element } \varphi \text{ of } {\rm H}_{\rm p}^{+}(S) \text{ such that }
 \#(\varphi)=N} \},$\\[0.5mm]
${\mathcal K}_{+}^{-}(N) = \{ S\in {\mathcal K}_{+}(N) \, |\,
{ \text{there is an element } \varphi \text{ of } {\rm H}_{\rm p}^{-}(S) \text{ such that }
 \#(\varphi)=N}\}, $
 \end{tabular}
 $$
where the operator $\#$ stands for the order.
With these notations
we define:
$$\begin{tabular}{ll}
$p_{+}(N) = \min \{ p(S) \,|\, S \in {\mathcal K}_{+}(N)\}$,\; & $p_{-}(N) = \min \{ p(S) \,|\, S \in {\mathcal K}_{-}(N)\}$\\[0.5mm]
$p_{+}^{+}(N) = \min \{ p(S) \,|\, S \in {\mathcal K}_{+}^{+}(N)\}$,\; & $p_{+}^{-}(N) = \min \{ p(S) \,|\, S \in {\mathcal K}_{+}^{-}(N)\}$
\end{tabular}
$$
and \\[-6mm]
$$p(N) = \min \{ p_{+}(N), p_{-}(N) \}$$
The calculation of the above five values is known as the minimal genus problem.
{A bordered surface $S$ is called {\it $N$-minimal\/} if $p(S)$ attains
$p(N)$, $p_{\pm}(N)$ or $p_{+}^{\pm}(N)$.}

\s
Another problem of a similar type is the maximum order problem
which consists in finding, for a given $p$, the maximal order of a
finite action on a bordered topological surface of algebraic genus
$p$. For $G=\zz_N$ we refine this problem by considering
$$\begin{tabular}{ll}
$N_{+}^{+}(p) = \max \{ N \,|\, S^+_p \in {\mathcal K}_{+}^{+}(N) \}$, &
$N_{+}^{-}(p) = \max \{ N \,|\, S^-_p \in {\mathcal K}_{+}^{-}(N) \}$\\[0.5mm]
$N_{+}(p) = \max \{N_+^+(p), N_+^-(p)\}$, & $N_{-}(p) = \max \{ N \,|\, S^-_p\in {\mathcal K}_{-}(N) \}$
\end{tabular}
$$
and \\[-6mm]
$$N(p) = \max \{N_+(p), N_-(p)\}$$
These problems, of minimal genus and maximal order, were solved in
\cite{BEGG}, and here we consider the question of uniqueness of topological type of self-homeomorphisms of maximal order and self-homeomorphisms acting on surfaces of
minimal genus.

\subsection{On topological type of cyclic actions of a given non-prime order on bordered orientable surfaces of minimal genus}\label{mg+}
\begin{thm}[\cite{BEGG}, Theorem 3.2.5]\label{3.2.5}
Let $N$ be a non-prime
odd integer and let $p$ be the smallest prime dividing $N$. Then
$$
p_+^+(N)=p_+(N)=
\begin{cases}
(p-1)\dfrac{N}{p} & {\rm if\ } p^2 \mid N\\[4mm]
(p-1)\dfrac{N-p}{p} & {\rm if\ } \! p^2 \not| \, N
\end{cases}$$
and the corresponding $N$-minimal surface has $1$ boundary component.
\end{thm}

\begin{cor}\label{min gen}
The action realizing $p_+^+(N)$ and $p_+(N)$ given in Theorem
$\ref{3.2.5}$ is unique up to topological conjugation if $p^2$ does
not divide $N$ and there are $p-1$ classes of such action in the
other case.
\end{cor}
\begin{pf}In the proof of Theorem  \ref{3.2.5} in  \cite{BEGG} it was shown that the minimum genus is realized just for an  NEC-group
$\Lambda$ having signature $(0;+;[p,q]; \{(\,)\})$, where $q=N$ if $p^2\mid N$,
 and otherwise $q=N/p$. We have $k=1$, and
$t=(p,q)$ is equal to $p$ and $1$ respectively.
By Theorem \ref{d21}, there are $\varphi(t)$ topological types of action corresponding to this signature.
\end{pf}

\begin{thm}[\cite{BEGG}, Theorem 3.2.6]\label{3.2.6}
Let $N\neq 2$ be an even integer not divisible by $4$. Then
$p_+^+(N)=p_+^-(N)=N/2-1$ . Moreover any $N$-minimal surface
from ${\mathcal K}_+^+(N)$ has one boundary component, whilst
any such surface from ${\mathcal K}_+^-(N)$ has $N/2$ boundary components.
\end{thm}
\begin{cor}
The actions realizing $p_+^+(N)$ and $p_+^-(N)$ given in Theorem
$\ref{3.2.6}$ are unique up to topological conjugation.
\end{cor}
\begin{pf}In the proof of Theorem \ref{3.2.6} in \cite{BEGG} it was shown that $\Lambda$ determining the minimal genera  must have  signature $(0;+;[2, N/2]; \{(\,)\})$
in the case of $p_+^+(N)$ and  $(0;+;[N/2]; \{(2,2)\})$ in the case of
$p_+^-(N)$. In the first case there is a unique class of
such action by Theorem \ref{d21}. In the second case the action is unique
by Theorem \ref{d12}.
\end{pf}

\begin{thm}[\cite{BEGG}, Theorem 3.2.7]\label{3.2.7}
Let $4$ divide $N$. Then
$p_+^+(N)=N/2, p_+^-(N)=N/2+1$.
 Moreover any $N$-minimal surface
from ${\mathcal K}_+^+(N)$ has one boundary component, whilst
any such surface from ${\mathcal K}_+^-(N)$ has
$2$ boundary components if $8$ divides $N$, and otherwise $4$ boundary components.
\end{thm}
\begin{cor}
The actions realizing $p_+^+(N)$ and $p_+^-(N)$ given in Theorem
$\ref{3.2.7}$ are unique up to topological conjugation.
\end{cor}
\begin{pf}Also here it was shown  in  \cite{BEGG} that  the action realising $p_+^+(N)$ is given just by an NEC-group $\Lambda$ with signature $(0;+;[2, N]; \{(\,)\})$
and so the action is unique by Theorem \ref{d21}.
In turn the signature $(1;-;[2]; \{(\,)\})$ is the unique  one realising $p_+^-(N)$
and so this action is unique by Theorem \ref{mb1ori}.
\end{pf}

\subsection{On topological type of
cyclic actions of a given non-prime   order on  bordered non-orientable surfaces of minimal genus}\label{mg-}
\begin{thm}[\cite{BEGG}, Theorem 3.2.8]\label{3.2.8}
Let $N$ be a non-prime odd integer and let $p$ be the smallest prime
dividing $N$. Then $p_-(N)= (p-1){N}/{p} +1$ and the corresponding
$N$-minimal surface has $1$ boundary component if $p^2$ divides $N$,
and $1$ or $p$ boundary components if $p^2$ does not divide $N$ and
both of these cases can actually occur.
\end{thm}

\begin{cor}
The actions realizing $p_-(N)$ given in Theorem $\ref{3.2.8}$ are
unique up to topological conjugation if $k=p$ and there are $(p-1)/2$ types of action for $k=1$.
\end{cor}
\begin{pf}In the proof of Theorem \ref{3.2.9} in
\cite{BEGG} it was shown that $\Lambda$ realizing the minimum
genus must have  signature $(1;-;[p]; \{(\,)\})$.
So the corollary follows from Theorem \ref{mb1nonori}.
\end{pf}

\begin{thm}[\cite{BEGG}, Theorem 3.2.9]\label{3.2.9}
Let $N\neq 2$ be even. Then $p_{-}(N)=N/2$ and any $N$-minimal
surface from ${\mathcal K}_{-}(N)$ is a projective plane with $N/2$ boundary components.
\end{thm}
\begin{cor}
The actions realizing $p_{-}(N)$ given in Theorem $\ref{3.2.9}$ is
unique up to topological conjugation.
\end{cor}
\begin{pf}Also here it was shown in \cite{BEGG},  that $\Lambda$ must be an NEC-group with the signature $(0;+;[N]; \{(2,2)\})$
and so our Corollary follows from Theorem \ref{d12}.
\end{pf}

\subsection{On topological type of actions of a prime order $N$ on surfaces of minimal genus} Observe that all results from the previous section concerning minimal genus were formulated and proved for $N$ being non-prime. For prime $N$, more general results concerning the minimum genus $p_+^+(N,k)$, $p_+^-(N,k)$ and $p_-(N,k)$  of surfaces  with specified number $k$ of boundary components
are given in \cite{BEGG}.
 These functions are periodic with respect to $k$, and so their knowledge obviously gives an effective way to solve the minimum genus problem by simply taking the minimum of $p_\ast^\ast(N,k)$ for varying $k$,
and in this way the problem was  solved in
 \cite{BEGG}.
 One can however calculate the minimum genus for a given prime $N$ directly or using results of the previous; which is more relevant for our purpose which is also  topological classification of actions realizing $p^\ast_\ast$.

\begin{prop}\label{min genus of N=2} We have $p_+^+(2) = p_+^-(2) = p_-(2) =2$. {The topological type of a $\mathbb{Z}_2$-action on a bordered surface of algebraic genus $2$ is determined by the surface and the quotient orbifold. Up to topological conjugacy there are:}
\begin{itemize}
\item
{$2$}  actions realizing $p_+^+(2)${\rm :} $1$  on $1$-holed torus and $1$  on  $3$-holed sphere,
\item
{$4$} actions realizing $p_+^-(2)${\rm :} $2$ on $1$-holed torus
and $2$ on  $3$-holed sphere,
\item
$8$ actions realizing \ $p_-(2)${\rm :}  $3$  on $2$-holed projective plane and $5$ on  $1$-holed Klein bottle.
\end{itemize}
\end{prop}

\begin{pf}We shall see that
$p_+^+(2)=p_+^-(2) =p_-(2)=2$ -  the smallest admissible genus. Since in such case $N=2>1=p-1$ all the possible involved signatures appear in Lemma
\ref{big o-r}. We are interested with the ones with the normalized area $1/2$ and we list all of them here for the reader's convenience:
$$
\begin{tabular}{rlrl}
$(1)$&$(0;+;[\;];\{(2,2,2,2,2,2)\})$, & $(2)$&$(0;+;[\;];\{(\,),(2,2)\})$,\\
$(3)$&$(1;-;[\;];\{(2,2)\})$, & $(5)$ &$(0;+;[2];\{(2,2,2,2)\})$,\\
$(6)$&$(1;-;[2];\{(\,)\})$ &$(8)$&$(0;+;[2];\{(\,),(\,)\})$,\\
$(9b)$&$(0;+;[2,2,2];\{(\,)\})$ &$(10b)$&$(0;+;[2,2];\{(2,2)\}).$
\end{tabular}
$$
Now the signature  (1) give rise to a  reflection of the $3$-holed sphere with the disk with $6$ corner  points as the orbit space
which is unique up to topological conjugacy  by Theorem \ref{d6}.
By Theorem \ref{ann2} the signature (2) provides four actions of $\mathbb{Z}_2$: on $1$-holed Klein bottle, $2$-holed projective plane and {orientation-reversing reflections} of $1$-holed torus and $3$-holed sphere.
By Theorem \ref{mb2}  the signature (3) provides  two   actions of $\mathbb{Z}_2$: on $1$-holed Klein bottle  and an orientation-reversing action on $1$-holed torus.
By Theorem \ref{d14} the signature (5) gives rise to one action  on $2$-holed projective plane.
By Theorem \ref{mb1ori} the signature  (6) does not provide any action on bordered orientable surface, whereas by Theorem \ref{mb1nonori} it gives rise to
 one action on $1$-holed Klein bottle.
By Theorem \ref{ann1-nonori}, the signature (8) gives rise to two actions on  $1$-holed Klein bottle and $2$-holed projective plane, whereas by Theorem \ref{ann1} it
gives rise to an orientation-preserving action on $3$-holed sphere.
By Theorem \ref{d31}, the signature (9b) gives rise to the  orientation-preserving action on  $1$-holed torus.
Finally the signature (10b) gives rise to the unique action  on $1$-holed Klein bottle by Theorem \ref{d221}.
Observe also that any two actions corresponding to different signatures are not topologically conjugate.
\end{pf}

Observe now  that for odd $N$, there are no surfaces admitting orientation reversing self-homeomorphisms
of order $N$, and so we have to  look only for $p_-(N)$ and $p_+^+(N)$, and classify topologically all actions realising them.

\begin{prop}\label{min genus of odd prime}
Let $N$ be an odd prime, Then
$p_+^+(N)=N-1$, and
 $p_-(N)=N$.
 Furthermore, in both cases the corresponding surface has $k=N$ or $k=1$ boundary components. {Up to topological conjugation, there are:
\begin{itemize}
\item $2$ actions of order $N$ on  the $N$-holed non-orientable surface of algebraic genus $p_-(N)$. The orbit spaces of these actions are  1-punctured  M\"obius band and   1-punctured annulus;
\item $3(N-1)/2$ actions of order $N$ on the $1$-holed non-orientable surface of algebraic genus $p_-(N)$; $N-1$ with a 1-punctured annulus  and $(N-1)/2$
with a 1-punctured M\"obius band as orbit spaces of the actions;
\item unique action of order $N$ on the $N$-holed orientable surface of algebraic genus $p_+^+(N)$;
\item  $(N-1)/2$  actions of order $N$  on the $1$-holed orientable surface of algebraic genus $p_+^+(N)$.
\end{itemize}
}
\end{prop}
\begin{pf}
Let $\theta\colon \Lambda \to \mathbb{Z}_N$ be a BSK-epimorphism defining an action of $\mathbb{Z}_N$ on a bordered surface.
Then $\Lambda$ has an empty period cycle in order to produce holes in the corresponding surface $X={\mathcal H}/\Gamma$ for $\Gamma = \ker \theta$. Now all periods in $\Lambda$, if exists,   are equal to $N$ and so $(0;+;[N,N];\{(\,)\})$ is the signature of $\Lambda$  with the minimal possible area here.

On the other hand Theorem
 \ref{d21} asserts that such an epimorphism  indeed exist, and the corresponding surface is orientable and   has $N$ or $1$ boundary components. Furthermore, up to topological equivalence, there is unique such action or $\lceil \psi(N)/2\rceil = (N-1)/2$ actions respectively.  This completes the part of the proof concerning $p_+^+(N)$ and also shows that for the study of  $p_-(N)$ and its attainments, we need to consider NEC-groups with  bigger area.

The second smallest area in this case have NEC-groups $\Lambda$ with signatures
$$
(1;-;[N];\{(\,)\})\;  {\rm and} \;  (0;+;[N];\{(\,),(\,)\})
$$
which indeed,  due to the Hurwitz-Riemann formula, concern actions of $\mathbb{Z}_N$ on surfaces of algebraic genus $p=N$. % which appear in  $p_-(N)$ in the Proposition.
In the first case, such action indeed exists  by Theorem \ref{mb1nonori}. Furthermore,  the corresponding surface has $N$ or $1$ boundary components and therefore, up to topological conjugacy, the corresponding action is respectively unique or there are $(N-1)/2$ topological classes of such actions.
The second signature   may realize $\mathbb{Z}_N$-actions  both on orientable (Theorem \ref{ann1}) and non-orientable (Theorem \ref{ann1-nonori})
surfaces and in the latter case either $k=N$ and the action is unique,  or $k=1$ and there are $N-1$ actions
up to topological conjugation
mentioned in Theorem  \ref{ann1-nonori}.
 \end{pf}

\subsection{On topological type of cyclic actions  of   maximal  order  on bordered surfaces of given algebraic genus}\label{mo}
\begin{thm}[\cite{BEGG}, Theorem 3.2.18]\label{3.2.18}
Let $p\geq 2$ be an integer. Then
\begin{itemize}
\item
$N_{-}(p)=2p$, \\[2mm]
\item
$
N_+^+(p)=
\begin{cases}
2(p+1) & {\rm if\ } p \; {\rm is \; even}\;\\
2p & {\rm if\ } p \; {\rm is \; odd} \;
\end{cases}$\\[2mm]
\item%[(iii)]
$
N_+^-(p)=
\begin{cases}
2(p+1) & {\rm if\ } p \; {\rm is \; even} \; \\
2(p-1) & {\rm if\ } p \; {\rm is \; odd} \; \;
\end{cases}$
\end{itemize}

\noindent
In particular
$$
\!\!\!\!\!
N_+(p)=
\begin{cases}
2(p+1) & {\rm if\ } p \; {\rm is \; even} \\
2p & {\rm if\ } p \; {\rm is \; odd}
\end{cases}
\;\;\;\; {\rm and} \;\;\ N(p)=
\begin{cases}
2(p+1) & {\rm if\ } p \; {\rm is \; even} \\
2p & {\rm if\ } p \; {\rm is \; odd}
\end{cases}
$$
\end{thm}

\begin{cor}\label{3.2.18 - cor}
All actions realizing the {solutions of the} maximum order problem described in
Theorem $\ref{3.2.18}$ are unique up to topological conjugacy.
\end{cor}
\begin{pf}By the proof of Theorem 3.2.18 in \cite{BEGG},
$N_{-}(p)$ for arbitrary $p$ and $N_+^-(p)$ for even $p$ are
realized by NEC-groups with signatures {$(0;+;[2p]; \{(2,2)\})$} and
{$(0;+;[p+1]; \{(2,2)\})$} respectively, and so these actions are unique
by Theorem \ref{d12}. Next, $ N_+^+(p)$ is realized by
signatures {$(0;+;[2,p+1]; \{(\,)\})$} and {$(0;+;[2,2p]; \{(\,)\})$} for $p$
even and odd respectively, {the corresponding surface has one boundary component}, and so these actions are unique by
Theorem \ref{d21}.
 Finally, $ N_+^-(p)$ for odd $p$  is
realized by signature $(1;-;[2]; \{(\,)\})$, {and the corresponding surface has $2$ boundary components if $4$ divides $p-1$ and $4$ boundary components otherwise}, so this action is
unique again by Theorem \ref{mb1ori}.
\end{pf}

\bigskip
\noindent

{\bf Acknowledgements.} The authors thank  the referees for suggestions  concerning the exposition and for comments and questions mentioned in the introduction.

\end{document}